\documentclass[11pt,letterpaper]{amsart}

\usepackage[final]{hyperref}
\usepackage{graphicx,upgreek}
\usepackage{amscd}
\usepackage{amsmath}
\usepackage{caption}
\usepackage{amsfonts}
\usepackage{amssymb}
\usepackage{mathrsfs}
\usepackage{multicol}
\usepackage{color}

\numberwithin{equation}{section}

\definecolor{darkgreen}{rgb}{0.09, 0.65, 0.27}
\definecolor{debianred}{rgb}{0.84, 0.04, 0.33}
\definecolor{orange}{rgb}{1.0, 0.5, 0.0}
\usepackage[margin=1in]{geometry}

\newcommand{\R}{\mathbb R}

\newcommand{\be}{\begin{equation}}
\newcommand{\ee}{\end{equation}}
\newcommand{\ben}{\begin{eqnarray*}}
\newcommand{\een}{\end{eqnarray*}}

\newcommand{\cb}{\color{blue}}
\newcommand{\cred}{\color{red}}
\newcommand{\cgr}{\color{darkgreen}}

\newtheorem{theorem}{Theorem}
\newtheorem{lemma}{Lemma}
\newtheorem{remark}{Remark}

\newtheorem{proposition}{Proposition}
\newtheorem{definition}{Definition}
\allowdisplaybreaks

\newcommand{\norm}[1]{\ensuremath{\left\|#1\right\|}}

\usepackage[normalem]{ulem}
\definecolor{DarkBlue}{rgb}{0,0.1,0.7}  
\definecolor{DarkGreen}{rgb}{0,0.5,0.1}

\newcommand\soutD{\bgroup\markoverwith
	{\textcolor{DarkGreen}{\rule[.5ex]{2pt}{1pt}}}\ULon}
\newcommand{\dx}{\ensuremath{\, {\rm d}x}}
\newcommand{\dy}{\ensuremath{\, {\rm d}y}}

\title[Liouville results for   $(p,q)$-Laplacian equations]{Liouville results for  $(p,q)$-Laplacian elliptic equations\\
with source terms involving  gradient nonlinearities}

\begin{document}

\author{Mousomi Bhakta, Anup Biswas}
\address{Department of Mathematics, Indian Institute of Science Education and Research Pune, Dr.\
Homi Bhabha Road, Pune 411008, India,\\
email: mousomi@iiserpune.ac.in, \, anup@iiserpune.ac.in}

\author{Roberta Filippucci}
\address{Dipartimento di Matematica e Informatica, Universit\'a degli Studi di Perugia, Via Vanvitelli 1, 06123 Perugia, Italy, email: roberta.filippucci@unipg.it}

\begin{abstract}
In this paper, we present a series of Liouville-type theorems for a class of nonhomogeneous quasilinear elliptic equations featuring reactions that depend on the solution and its gradient. Specifically, we investigate equations of the form $-\Delta_p u - \Delta_q u = f(u,\nabla u)$ with  $p > q > 1$, where the nonlinearity $f$ takes forms such as $u^s|\nabla u|^m$ or $u^s + M|\nabla u|^m$ ($s,\,m\geq 0$).

    Our approach is twofold. For cases where the reaction term satisfies $|f(u,\nabla u)| \leq g(u)|\nabla u|^m$ with $m > q$ and $g$ continuous, we prove that every bounded solution (without any sign restriction) in $\mathbb{R}^N$ is constant by means of an Ishii–Lions type technique. In the remaining scenarios, we turn to the Bernstein method. The application of this method to the nonhomogeneous operator requires a nontrivial adaptation, as, roughly speaking, constant coefficients are replaced by functions that may not be bounded from above, which enables us to establish a crucial a priori estimate for the gradient of solutions in any domain $\Omega$. This estimate, in turn, implies the desired Liouville properties on the entire space $\mathbb{R}^N$.
     As a consequence, we have fully extended Lions Liouville-type result for the Hamilton-Jacobi equation to the 
$(p,q)$-Laplacian setting, while for the  
$(p,q)$ generalized Lane-Emden equation, we provide an initial contribution in the direction of the classical result by Gidas and Spruck for $p=q=2$, as well as that of Serrin and Zou for $p=q$.

    To the best of our knowledge, this is the first paper which studies Liouville properties for equations with nonhomogeneous operator involving source gradient terms.

\end{abstract}


\maketitle
\setcounter{tocdepth}{1}
\tableofcontents

\medskip

\noindent

\section{Introduction}

In this paper we obtain Liouville type results for positive solutions to the following equation
\begin{equation}\label{main}-\Delta_pu-\Delta_q u= f(u,\nabla u) \quad\mbox{in}\quad \mathbb R^N, \qquad p>q>1, \end{equation}
where  the source term $f$, defined in $\mathbb R^+\times\mathbb R^N$ ($N\geq 2$), has three different forms
$$\begin{aligned}f(t,\xi)&=|\xi|^m,\quad m>p-1,\\
f(t,\xi)&=t^s|\xi|^m, \quad s>0, \,\, m\ge0,\\
f(t,\xi)&=t^s+M|\xi|^m,\quad M>0, \, s,m>0 .\end{aligned}$$
 We point out that, when $m=0$ in  the second expression of $f$, our Liouville results include also Lane -Emden type nonlinearities, investigated in the Laplacian case in the pioneering paper \cite{GS} by Gidas and Spruck and later extended to the $p$-Laplacian operator by Serrin and Zou in \cite{SZ}.

Equation \eqref{main} is driven by the  $(p,q)$-Laplacian operator, given by a combination of two $s$-Laplacian operators, arising in many applications such as the study of reaction-diffusion systems whose general  form is 
$$u_{t}=\mbox{div}[A(u)\nabla u]+c(x,u),$$
where the function $u$ is a state variable and describes the density or concentration of multicomponent substances, while $A(u)$ is called
diffusion coefficient, the term $c(x,u)$
 is the reaction and relates to sources and loss processes.  Typically, in chemical and biological applications, the reaction
 term $c(x,u)$ has a polynomial form with respect to the concentration $u$.
 The $(p,q)$-Laplacian  operator can be obtained for a diffusion coefficient having a power
 law dependency of the form $A(u)=(|\nabla u|^{p-2}+|\nabla u|^{q-2})$.
 Reaction-diffusion systems  have a wide range of applications in physics and related sciences, such as biophysics, chemical reaction and plasma physics.  The initial approach to handling such operators originates from Zhikov \cite{Z} (see also \cite{M}), who introduced these classes in the context of modeling strongly anisotropic materials.

Another remarkable  subcase of \eqref{main} is the nonlinear Schr\"odinger equation, which allows to study solitary waves or solitons, which are special solutions whose profile remain unchanged under the evolution in time. Here we are interested in the stationary version.

When dealing with gradient type nonlinearities,  for models used in population dynamics, we refer to \cite{S}, see also \cite{B} where the term $u^s|\nabla u|^m$ is interpreted in terms of a probability function, modelling the predatory greed during a predation event.

The study of equation \eqref{main} presents several challenges. One major difficulty arises from the structure of the differential operator involved, which is not only nonlinear and obtained as a combination of possibly degenerate and singular operators, but also nonhomogeneous, precluding the application of  well-established techniques traditionally used in the homogeneous setting.  Moreover, the presence of a gradient term in the nonlinearity further complicates the analysis: it prevents the problem from being variational in nature and requires the use of sophisticated techniques to address it. 

A first physical model for gradient type nonlinearities is given by the Hamilton-Jacobi equation $\Delta u=|\nabla u|^m$ in $\Omega\subset\mathbb R^N$ 
first  investigated by Lions in \cite{Lions}.
Using a Bernstein-type technique, he established the {\it Liouville property}, that is,  any $C^2$ solution must be a constant,  for every $m>1$. 
The quasilinear version of the Hamilton-Jacobi  equation 
 $$\Delta_p u=|\nabla u|^m \quad  \mbox{in} \quad  \Omega,$$ was later investigated
by Bidaut-Véron, Garcia-Huidobro and Véron in \cite{BV_HJ}, obtaing that for any $C^1$ solution  in an arbitrary domain $\Omega\subset\mathbb R^N$, with  $1<p\le N$  and $m>p-1$, the following estimate holds
$$|\nabla u(x)|\le C (\mbox{dist}(x,\partial\Omega))^{1/(m-p+1)}.$$
As a consequence,  a Liouville-type result holds when $\Omega=\mathbb R^N$.  This result is in the same spirit as the work of Dancer \cite{D}, and it is also related to the findings in \cite{SZ}. 

A further generalization considers a reaction term that depends not only on the gradient, but also explicitly on a power of $u$, namely
\begin{equation}\label{eqgh}
-\Delta_p u=u^s|\nabla u|^m \quad  \mbox{in} \quad  \mathbb R^N,
\end{equation}
introduced in its radial form for  $p=2$ in \cite{CM}. 
It is well known that any nonconstant, nonnegative supersolution to equation \eqref{eqgh} must in fact be constant in the so-called {\it first subcritical range}, defined by
\begin{equation}\label{supercr}
s(N - p) + m(N - 1) \leq N(p - 1),
\end{equation}
for which we refer to \cite{MP} and \cite{F} for further details.

Bidaut-Véron showed in \cite{BV}  that when $1<p<N$, $s\ge0$ and $m\ge p$, any positive $C^1$ solution to \eqref{eqgh}
must be constant, generalizing a previous work by Filippucci, Pucci and Souplet in \cite{FPS}, for the case $p=2$ and assuming boundedness of the solution.  

In the case \eqref{eqgh} with 
$m<p$, Liouville-type results are known only for certain subregions. We refer to \cite{VBVGH} in the case 
$p=2$, where Theorem B establishes the Liouville property as a consequence of pointwise gradient estimates in arbitrary domains $\Omega\subset\mathbb R^N$. These estimates are obtained using a direct Bernstein method combined with a change of variables, resulting in an a priori estimate for a suitably chosen auxiliary function.
An initial extension to the 
$p$-Laplacian is presented in \cite{CHZ}, where the authors introduce a technical device to circumvent the change of variables—an approach that would otherwise entail significant algebraic complexity due to the nonlinear structure of the 
$p$-Laplacian. This alternative method, however, leads to a slightly more restrictive threshold; see Remark 1.1 in \cite{CHZ} for details.

Concerning the Liouville property for positive solutions of 
\begin{equation}\label{sum}
-\Delta_p u=u^s+M|\nabla u|^m \quad  \mbox{in} \quad  \mathbb R^N,\quad M>0,
\end{equation}
the main contributions can be found in \cite{VB2} when $p=2$ and in \cite{FSZ} for equation \eqref{sum}, where again the direct method of Bernstein is employed. As discussed in details in \cite{VB2}, the equation \eqref{sum}  presents some similarities with either the Lane–
Emden equation or the Hamilton-Jacobi equation, depending on whether the exponent $m$ is  subcritical or supercritical with respect to $ps/(s+1)$.

Equations \eqref{eqgh} and \eqref{sum} have a common feature that they are invariant under the action of transformations of the form 
$$T_\sigma[u](x)=\sigma^\alpha u(\sigma x), \qquad \sigma, \alpha>0,$$
with  $\alpha=\frac{p-m}{s+m-p+1}>0$ and $\alpha=\frac{ps}{s+1}$, respectively.

In dealing with the $(p,q)$-Laplacian, the lack of homogeneity requires a delicate extension of the Bernstein technique. Indeed, instead of constant coefficients, one now has to handle functions that depend on the solution $u$  and on $|\nabla u|^2$, which are not only variable but also unbounded from above. As a result, highly nontrivial estimates are needed when $m\le q$.

Moreover, when $m>q$, the presence of these functions prevents the derivation of upper estimates, thereby making it is impossible to apply a Bernstein-type technique. For this reason, a different approach, based on the Ishii–Lions method \cite{IL90} and discussed below, is employed.

\medskip 
In this paper, as in \cite{SZ}, we consider weak  solutions of \eqref{main}, namely,
 \begin{definition}\label{defi-sol}
We say that a  function $u\in C^1(\Omega)$ is a solution  of  $-\Delta_p u-\Delta_q u= f(x,u,\nabla u)\quad\text{in  }\, \Omega$ if 
$$\int_{\Omega}|\nabla u|^{p-2}\nabla u\cdot\nabla \varphi \dx+\int_{\Omega}|\nabla u|^{q-2}\nabla u\cdot\nabla \varphi dx= \int_{\Omega}f(x,u,\nabla u)\varphi dx\quad\forall\,\, \varphi\in C^1_c(\Omega).$$
\end{definition}
 Throughout this article, by a subsolution, supersolution or a solution we would mean $C^1$ weak
subsolution, supersolution and solution, respectively.

We begin by presenting the main results of the paper, which address all three nonlinearities in \eqref{main} within $\Omega$, where $\Omega$ is a domain in 
$\mathbb{R}^N$ with $N\geq 2$. 
The first result is the complete extension of Lions result for the Laplacian in \cite{Lions} and that of Bidaut Veron et al. \cite{BV_HJ} for the $p$-Laplacian, to the Hamilton-Jacobi involving $(p,q)$-Laplacian case.  

\begin{theorem}\label{th_HJ} 
    Let $\Omega \subset \mathbb R^N$ be a domain, and assume $m > p-1$. Let $u$ be a solution to
    \begin{equation}\label{main_HJ}-\Delta_p u - \Delta_q u = |\nabla u|^m \quad \text{in } \Omega.
    \end{equation}
    Then the following hold:
    \begin{enumerate}
    \item[(i)] There exists a positive constant $C = C(N, m, p, q)$ such that
    \begin{equation}\label{1-8}
    \left|\nabla u(x)\right| \leq C \left( 1 + \mathrm{dist}(x, \partial\Omega)^{-\frac{1}{m-p+1}} \right) \quad \forall\, x \in \Omega.
    \end{equation}
    \item[(ii)] If $\Omega = \mathbb R^N$, then every solution of \eqref{main_HJ} is constant.
    \end{enumerate}
\end{theorem}
 \begin{remark} It is clear from the statement that, in the Hamilton–Jacobi 
$(p,q)$-Laplacian equation, the leading term in the differential operator is actually the higher-order one, namely the 
$p$-Laplacian. 

Moreover, we emphasize that the Liouville property \mbox{\rm (ii)} does not follow directly from \mbox{\rm (i)}; 
rather, its proof requires an additional deduction and the conclusion then follows through a limiting procedure.
\end{remark}

In  the next result, we consider the second type of nonlinearity, namely, the ``product" one.
To this aim, we make use of the following threshold values

\begin{equation}
\label{defQ2}\mathcal Q_1:=
\frac{2(q-1)}{N}-\sqrt{\frac{4(q-1)^2}{N^2}-\mathcal{R}},
\end{equation}
\begin{equation}
\label{defQ1}\mathcal Q_2:=
\frac{2(q-1)}{N}+\sqrt{\frac{4(q-1)^2}{N^2}-\mathcal{R}},
\end{equation}
and
\begin{equation}
\label{defQ3}\mathcal Q_3:=\mathcal Q_2+\frac{\bigl[\frac2N(q-1)-s\left(\frac2N (p-1) +p-q\right)\bigr]^2}{s\bigl[\frac{s\mathcal R}{\mathcal Q_2}+\frac4N(p-1)\bigr]},
\end{equation}
where
\begin{equation}\label{R}
\mathcal{R}:=(p-q)\bigl[p-q+\frac4N(p-1)\bigr].
\end{equation}

The expressions for these values highlight the significant complexity introduced by the nonhomogeneous nature of the operator.

\begin{theorem}\label{th_sm}
Let $\Omega\subset\mathbb R^N$,  $s>0$, $m\geq 0$,
\begin{equation}\label{lim_beta_2_pos} 
1-\frac{(p-q)(1+s)}{m+s-q+1}>0,
\quad\ m+s>p-1 \quad\mbox{and}\quad 4(q-1)^2\geq N^2\mathcal{R},
\end{equation} 
where $\mathcal{R}$ is defined by \eqref{R}. 
Denote,
$$\mathcal{Q}:=m+s-q+1.$$
Suppose one of the following assumptions holds
\begin{enumerate}
    \item[\rm(A)] $\mathcal Q_1<\mathcal{Q}<\mathcal Q_2$,
 \vspace{0,2cm}
 \item[\rm(B)] $\mathcal Q\in\{\mathcal Q_1, \mathcal Q_2\}$ and 
 $s<\dfrac{q-1}{p-1+\frac N2(p-q)}$, \vspace{0,2cm} {}
 \item[\rm(C)] 
 $s<\dfrac{q-1}{p-1+\frac N2(p-q)}$,  $ 0\leq m\leq q<p<m+1$ and
$$\mathcal{Q}\in\begin{cases} (\mathcal Q_2, \mathcal Q_3)\cup\bigg(\frac{N\big((1-a)\mathcal Q_1^2+\mathcal R\big)}{4(q-1)},\, \mathcal Q_1\bigg) &\quad\mbox{if}\quad a\leq 1,\\
(\mathcal Q_2,\mathcal Q_3)\cup \bigg(\frac{\mathcal R N}{4(q-1)},\,\mathcal Q_1\bigg) &\quad\mbox{if}\quad a>1,\end{cases}$$
 \end{enumerate}
where $\mathcal{Q}_1$, $\mathcal{Q}_2$, $\mathcal{Q}_3$ are given by \eqref{defQ1}-\eqref{defQ3}, respectively, and 
$$a=\frac{N}{s}\frac{\bigg[\frac{2}{N}(q-1-s(p-1))-s(p-q)\bigg]^2}{Ns\mathcal R+4(p-1)\mathcal Q_1}.$$
Then, the following hold:

\smallskip
 
(i) There exist positive constants $b, \gamma$ and $C=C(N,m,p,q,s)$,
 with
 \begin{equation}\label{b>0}
b>\max\biggl\{ 0, \frac{m-q+1}{m+s-q+1}\biggr\},
\end{equation}
such that any positive solution of 
\begin{equation}\label{main_sm}
-\Delta_pu-\Delta_q u= u^s|\nabla u|^m\quad\mbox{in}\quad \Omega, 
\end{equation} 
satisfies
\begin{equation}\label{1-8}
\left|\nabla u^{1/b}(x)\right|\leq \left(1+ C
\left({\rm dist}\left(x,\partial\Omega\right)\right)^{-\frac{2}{\gamma}}\right) \quad\forall \, x\in\Omega.
\end{equation}

(ii) Every nonnegative solution of \eqref{main_sm} is constant in $\mathbb R^N$.
\end{theorem}

\begin{remark} We point out that in the case of $p$-Laplacian, i.e., $p=q$ in \eqref{main_sm},  
$$ \mathcal Q_1=0,\quad \mathcal Q_2= \frac{4(p-1)}N, \quad  \mathcal Q_3=\frac{p-1}N\frac{(1+s)^2}s\quad \mbox{and} \quad \mathcal R=0.$$
In addition, the threshold for $s$ becomes exactly  1, so that  Theorem \ref{th_sm} reduces to Theorem 1.1 in \cite{CHZ}. 
 Furthermore, in the subcase $p=q=2$ Theorem \ref{th_sm}-(A) reduces to Theorem B-(i) in \cite{BV}, while  Theorem \ref{th_sm}-(ii) exhibits a smaller range for $m+s$, as emphasized in \cite{CHZ}, since, due to the complexity of the technique even in the $p$-Laplacian case, a second change a variable used in \cite{BV} has been avoided.

  Furthermore, we emphasize that the lower bounds for $\mathcal{Q}$ in Theorem $\ref{th_sm}$-$(C)$, in both cases $a \le 1$ and $a > 1$, are not optimal. They are introduced primarily to keep the statement concise, while still highlighting the appearance of a new range for $\mathcal{Q}$ when $p \ne q$, given by $(0,\mathcal Q_1)$. The optimal thresholds will become clear during the proof through a straightforward calculation.
  
Finally, when $m=0$, equation \eqref{main_sm} reduces to the $(p,q)$ generalized Lane-Emden equation. Thus,  to the best of our knowledge, this is a first result involving the (p,q)-Laplacian operator, in the direction of a Gidas–Spruck type theorem, as well as a result in the spirit of Serrin and Zou. 
 On the other hand, in the subcase $p=q=2$, Theorem \ref{th_sm}-(A) holds for $s<(N+4)/N$ which is smaller than the critical Sobolev exponent $2^*$ for $N>2$.

\end{remark}

Our next theorem addresses \eqref{main_sm} when $m>q$. In this case, if we follow a Bernstein method,
 the polynomials obtained in the process do not have constant coefficients (see for instance, \eqref{trin}), but functions as coefficients and these functions are not necessarily bounded. This creates a hurdle in adapting a Bernstein-type estimate similar to \cite{BV} to prove $u^{1+s}|\nabla u|^{m-q}$ to be bounded.  When $p=q$, Bernstein estimate was obtained in \cite{BV} and Liouville property was then established by using (scale free) weak-Harnack property for the superharmonic functions and half-Harnack inequality for an appropriate power of $(u-l)$, where $l$ is an appropriate constant, see \cite{BV} more details. In our set-up, these Harnack type estimates,  especially for inequalities, seem quite challenging and are not covered by the existing literature. Furthermore, our operator is not scale free due to its nonhomogeneity.
We therefore adopt a completely different strategy and, as a first attempt in the literature, restrict our attention to bounded solutions. The proof relies on an argument of Ishii–Lions type, originally introduced in \cite{IL90} to establish H\"{o}lder regularity of viscosity solutions for nondegenerate elliptic second-order equations. For an application of this method to nonlocal operators, we refer to the recent papers of Barles et al. \cite{BCI, BCCI}. This technique typically involves doubling the variables and introducing a penalization function that serves as a test function for the solution. In contrast to the standard Ishii–Lions method, our argument requires the H\"{o}lder constant of this test function to be sufficiently small. Together with the ellipticity of the equation, this condition yields the desired result.
A similar idea was employed in the context of nonlocal operators in \cite{BQT}.

\begin{theorem}\label{AB001}
All bounded solutions to $-\Delta_pu-\Delta_q u=f(u, \nabla u)$ in $\mathbb{R}^n$ are constants where
\begin{equation}\label{f_IL} |f(u, \nabla u)|\leq g(u) |\nabla u|^m \quad \text{and}\quad m>q >1,
\end{equation}
for some continuous function $g$.
\end{theorem}

\begin{remark}
Note that the conclusion in  Theorem~\ref{th_sm}(A)-(B) also
hold for $m>q$ and does not need any boundedness assumption on the solution. On the other hand, \eqref{lim_beta_2_pos}  in Theorem~\ref{th_sm} restricts $p-q$ from being arbitrarily large.
\end{remark}

\medskip
In the next two theorems we focus on the nonlinearities which are sum of $u^s$ and $|\nabla u|^m$.  In the first theorem we obtain an estimate for the growth of any solution, in the second the Liouville property is reached.

\begin{theorem}\label{thm+}
Let $\Omega\subset\mathbb R^N$. Assume $m-p+2>0$ and
\begin{equation}\label{bound+}
s>\max\{ q-1,1\}, \qquad m>\max\biggl\{\frac{qs}{s+1}, 2s\biggr\}.
\end{equation}

Then, for any $M>0$, there exists a positive constant $C=C(N,m,p,q,M)$ such that any positive solution of \begin{equation}
\label{main_M+}-\Delta_pu-\Delta_q u= u^s+M|\nabla u|^m\quad\mbox{in}\quad \Omega, 
\end{equation} 
satisfies
\begin{equation}\label{1-2}
\left|\nabla u(x)\right|\leq C\left(1+
{\rm dist}\left(x,\partial\Omega\right)^{-\frac{1}{m-p+2}}\right)\end{equation}
for all  $x\in\Omega$.
Especially, any positive solution of \eqref{main_M+} in $\mathbb R^N$ has at most a linear growth at infinity, being in force
\begin{equation}\label{lin_growth}
\left|\nabla u(x)\right|\leq C
\quad x\in\mathbb R^N.
\end{equation}
\end{theorem}
 \begin{theorem} \label{thm+_v^b} Let $\Omega\subset\mathbb R^N$.
Assume $N(p-q)<2(q-1)$,
\begin{equation}\label{sqrt_cond}\Delta_{p,q}:=(N+2)^2(q-1)^2-N(N+4)(p-1)^2-4N(p-q)^2>0.\end{equation}
Define
\begin{equation}\label{S-}
S_-:=\frac{(N+2)(q-1)-\sqrt{\Delta_{p,q}}}N \end{equation}
and
\begin{equation}\label{S+}
S_+:=\frac{(N+2)(q-1)+\sqrt{\Delta_{p,q}}}N\end{equation}
Assume
\begin{equation}\label{s_bound_new}
\max\{S_-,p-1\}<s<S_+,\end{equation}
\begin{equation}\label{beta_2_pos_sum} 1-\frac{(p-q)(1+s)}{s-q+1}>0\end{equation}
and
\begin{equation}\label{m_bound_new} 0<m\le \frac{N+2}N(q-1)
\end{equation}
Then,  there exist positive constants $b$, $\gamma$ and $C=C(N,m,p,q,s)$,
such that any positive solution of 
\eqref{main_M+} satisfies
\eqref{1-8}.

In addition, every nonnegative solution of \eqref{main_M+} is constant in $\mathbb R^N$.
\end{theorem}
\begin{remark}  We observe that when $p=q$,  condition \eqref{sqrt_cond} is automatically satisfied since $\Delta_{p,p}=4(p-1)^2>0$. As a consequence, we obtain
$S_-=p-1$, yielding the same lower bound for $s$ in the case of $p$-Laplacian (see \cite[Theorem 1.9]{FSZ}),  while $S_+=\frac{N+4}N(p-1)$, which is slightly larger than the threshold $\frac{N+3}{N-1}(p-1)$ found in the same setting in \cite{FSZ}.  Moreover, still in the case 
$p=q$, the threshold appearing in \eqref{m_bound_new} coincides with that of the  $p$-Laplacian case discussed in \cite{FSZ}.
\end{remark}

\smallskip
Very recently in \cite{BBF}, the authors of this paper have studied Liouville  properties of various differential inequalities of the form
$$-\Delta_p u-\Delta_q u\geq f(u,\nabla u) \quad\text{in}\quad\Omega,$$
where $\Omega$ is any exterior domain. In particular, Liouville properties of supersolutions to \eqref{main_sm} have been discussed in \cite{BBF} for  $s\geq 0$, $m\geq 0$. 

\medskip

This paper is organized as follows. Section 2 provides preliminary material and establishes elementary results used throughout the subsequent sections. In Section 3, we prove Theorem~\ref{th_HJ}. Sections 4 and 5 address equations with product nonlinearities, containing the proofs of Theorem~\ref{th_sm} and Theorem~\ref{AB001}, respectively. Finally, Section 6 discusses equations involving the sum of nonlinearities, proving Theorem~\ref{thm+} and Theorem~\ref{thm+_v^b}.

\section{Preliminary results}

Any solution, as defined in Definition~\ref{defi-sol}, to the equation $-\Delta_p u-\Delta_q u =f(u,\nabla u)$ in $\Omega$ with $f\in L^\infty(\Omega)$ is known to be in $C^{1, \alpha}_{\rm loc}(\Omega)$, 
see \cite[Theorem 1.7]{Lib91}. Thus, any $C^{1}$ solution of $-\Delta_p u-\Delta_q u= g(u) |\nabla u|^m$ is in
$C^{1, \alpha}_{\rm loc}$ for $g$ continuous and for any $m>0$. In addition, suppose that $g\in C^{k, \gamma}_{\rm loc}(\mathbb{R})$ for some $k\in\mathbb{N}$, $\gamma\in (0,1)$ and consider a ball
$B_r(x_0)$ such that $|\nabla u|>0$ in $\overline{B_r(x_0)}$. Then, from \cite[Theorem~4.5.2]{LU68},  we obtain 
$u\in W^{2,2}_{\rm loc}(B_r(x_0))$. More precisely, if we set $a_i(\sigma) =|\sigma|^{p-2}\sigma_i + |\sigma|^{q-2}\sigma_i$
and $a(u, \sigma)=g(u) |\sigma|^m$, for $\sigma\in\mathbb{R}^n$, we have
\begin{align*}
 0<\lambda |\xi|^2 \leq   \sum_{i, j}\frac{\partial a_i(\nabla u(x))}{\partial\sigma_j}\xi_i\xi_j &\leq \Lambda |\xi|^2,
 \\
 \sum_{i=1}^N \left(|a_i(\nabla u)|+ \Bigl|\frac{\partial a(u(x), \nabla u(x))}{\partial \sigma_i}\Bigr|\right)(1+|\sigma|) + |a(u, \nabla u)|
 &+ \Bigl|\frac{a(u(x), \nabla u(x))}{\partial u}\Bigr| \leq \Lambda (1+ |\sigma|^2)
 \end{align*}
  in $B_r(x_0)$, for some constants $\lambda, \Lambda$. These two conditions are enough to apply \cite[Theorem~4.5.2]{LU68}, giving us
$u\in W^{2,2}_{\rm loc}(B_r(x_0))$. Now we can apply \cite[Theorem~4.6.3]{LU68}, with a similar reasoning as above
(see the discussion on page 282 of \cite{LU68}), to conclude that
$u\in C^{k+2}(B_r(x_0))$.

\medskip

Before proceeding further, in order to simplify the notation, we introduce three functions which will play a crucial role in the proofs below.
Precisely, from now on let
$$\begin{aligned}&A:=1+|b|^{p-q}v^{(b-1)(p-q)}z^{(p-q)/2},\\
&D:=q-2+(p-2)|b|^{p-q}v^{(b-1)(p-q)}z^{(p-q)/2},\\
&E:=q-1+(p-1)|b|^{p-q}v^{(b-1)(p-q)}z^{(p-q)/2},
\end{aligned}$$
so that
\begin{equation}\label{estim_DA_EA} A\ge1, \qquad 0\le q-2\le\frac DA\le p-2,\qquad q-1\le\frac EA\le p-1.\end{equation}
In particular, it holds  
$$E+p-q=(p-1)A\qquad\text{and}\qquad D+p-q=(p-2)A,$$
and in the special case $p=q$ the above functions reduce to $A=2$, $E=2(p-1)$ and $D=2(p-2)$.
\begin{lemma}\label{lem1}
\rm{ Let $\Omega\subset\mathbb R^N$, $N\geq1$ and $m>1$. Assume that $v$ is continuous, $|\nabla v|>0$, and $w$ is continuous and nonnegative in $\Omega$ and $C^1$ on the set $\mathcal{W}_+=\left\{x\in\Omega:w(x)>0\right\}$. Define the operator
$$w\rightarrow\mathscr{A}_v(w):=-\Delta w-\frac DA\frac{\left<D^2w\nabla v,\nabla v\right>}{|\nabla v|^2}.$$
If $w$ satisfies, for some $\xi>1$, a constant $C>0$ and a real number $c_0$,
$$\mathscr A_v(w)+Cw^\xi\leq c_0\frac{\left|\nabla w\right|^2}{w}$$
on each connected component of $\mathcal{W}_+$, then
$$w(x)\leq c_{N,\xi,c_0}\left({\rm dist}\left(x,\partial\Omega\right)\right)^{-\frac{2}{\xi-1}},\quad\forall x\in\Omega.$$
In particular, $w\equiv0$ if $\Omega=\mathbb R^N$.
}
\end{lemma}
\begin{proof} The proof follows from  a combination of  \cite[Lemma 2.2]{VBVGH} and \cite[Lemma  3.1]{BV}, the latter with $\beta=0$ and $\alpha=C$ (also see \cite[Lemma 2.1]{FSZ}). Indeed, it is enough to observe that the operator
$$z\to\mathscr A_v(w)=\sum_{i=1}^N\biggl(\delta_{i,j}+\frac DA \frac{v_{x_i}v_{x_j}}{|\nabla v|^2}\biggr)w_{{x_i}x_j}(z):=\sum_{i=1}^N a_{i,j}w_{{x_i}x_j}(z),$$
is uniformly elliptic, indeed
thanks to \eqref{estim_DA_EA} we have $$\min\{1,q-1\}|\xi|^2\le \sum_{i=1}^N a_{i,j}\xi_i\xi_j\le \max\{1,p-1\}|\xi|^2$$
for all $\xi\in\mathbb R^N$ and $1<q<p$.
\end{proof}
\begin{lemma}\label{lemma_Bcal} Let $u$ be a nonnegative solution of \eqref{main},  let $v=u^{1/b}$ with  $b\in\mathbb R\setminus\{0\}$ and 
$z:=|\nabla v|^2$. Denote with $\mathcal Bu=\Delta_pu+\Delta_qu$.

Then, the following inequality holds 
$$\begin{aligned} \frac12\mathscr{A}_v&(z) +\frac {1}{Nb^{2(q-1)}}\frac{\bigl(\mathcal Bu\bigr)^2}{A^2}
\frac1{v^{2(b-1)(q-1)}z^{q-2}}\\
&\le 
\frac{1}{b|b|^{q-2}}\frac{\mathcal Bu}{A}\frac{z^{-q/2}}{v^{(b-1)(q-1)}}\biggl(\frac 1N+\frac 12\biggr)\frac DA\langle\nabla z,\nabla v\rangle\\
& 
\quad+2\frac{b-1}{b|b|^{q-2}}\frac{\mathcal Bu}A\frac1{v^{(b-1)(q-1)}z^{(q-2)/2}}\biggl(\frac 1N+\frac 12\biggr)\frac EA\frac zv
\\
&\quad -\frac {1}{b|b|^{q-2}}\frac{\langle\nabla\bigl(\mathcal Bu\bigr),\nabla v\rangle}{A}
\frac{1}{v^{(b-1)(q-1)}z^{(q-2)/2}}
\\
&\quad+(b-1)\biggl[-\frac 1N\frac{DE}{A^2}+{\frac12}\frac{(p-q)^2}A\frac{A-1}A+\frac EA\biggr]\frac{\langle\nabla z,\nabla v\rangle}v
\\
&
\quad+(b-1)\biggl[-\frac{b-1}N\frac{E^2}{A^2}+(b-1)\frac{(p-q)^2}A\frac{A-1}A- \frac EA\biggr]\frac{z^2}{v^2}\\
&\quad -\frac 12\biggl[{ -}\frac12\frac{(p-q)^2}A\frac{A-1}A+ \frac DA+\frac1{2N}\frac{D^2}{A^2}\biggr]\frac{\langle\nabla z,\nabla v\rangle^2}{z^2}+\frac14\frac DA\frac{|\nabla z|^2}{z}\quad \mbox{on}\quad \{z>0\}.
\end{aligned}$$

\end{lemma}
\begin{proof}
By the definition of $v$ we have
$$\begin{aligned}\Delta_pu
 =b|b|^{p-2}\,\bigl\{v^{(b-1)(p-1)}&[|Dv|^{p-2}\Delta v+(p-2)|\nabla v|^{p-4}\left<D^2v\nabla v,\nabla v\right>]\\&\qquad \qquad+
(b-1)(p-1)v^{(b-1)(p-1)-1}|Dv|^p\bigr\},
\end{aligned}$$
so that, replacing $|\nabla v|^2=z$, we get
$$\Delta_pu=b|b|^{p-2}\,v^{(b-1)(p-1)}z^{-1+p/2}
\bigl[\Delta v+(b-1)(p-1)zv^{-1}
+(p-2)z^{-1}\left<D^2v\nabla v,\nabla v\right>\bigr].$$
Consequently, using that $2\left<D^2v\nabla v,\nabla v\right>=\left<\nabla z,\nabla v\right>$, we have
\begin{align}\label{Aug-20}
\Delta_pu+\Delta_qu&
=b\biggl[|b|^{p-2}v^{(b-1)(p-1)}z^{-1+p/2}+|b|^{q-2}v^{(b-1)(q-1)}z^{-1+q/2}\biggr]\Delta v \nonumber
\\ 
&\quad 
+b(b-1)\biggl[|b|^{p-2} (p-1)v^{(b-1)(p-1)-1}z^{p/2}+|b|^{q-2}(q-1) v^{(b-1)(q-1)-1}z^{q/2}\biggr]
\nonumber
\\&\quad +\frac b2\biggl[(p-2)|b|^{p-2}v^{(b-1)(p-1)}z^{p/2}+(q-2)|b|^{q-2}v^{(b-1)(q-1)}z^{q/2}\biggr]
\frac{\bigl<\nabla z,\nabla v\bigr>}{z^2}
\nonumber\\
&=b|b|^{q-2}v^{(b-1)(q-1)}\biggl[z^{-1+q/2}A\,\Delta v
+(b-1)Ez^{q/2}v^{-1}+\frac D2z^{(q-4)/2}
\langle\nabla z,\nabla v\rangle\biggr]
\end{align}
yielding the following expression for $\Delta v$ being $u$ a solution of  \eqref{main}
 \begin{equation}\label{lapla}
\Delta v=\frac {1}{b|b|^{q-2}}\frac{\mathcal Bu}{A}
\frac{z^{1-q/2}}{v^{(b-1)(q-1)}}-
(b-1)\frac EA\frac zv-\frac12 \frac DA\frac{\bigl<\nabla z,\nabla v\bigr>}{z}.
    \end{equation}
A routine calculation gives
$$\begin{aligned}\nabla_v\biggl(\frac EA\biggr)&=-\frac1{A^2}|b|^{p-q}(b-1)(p-q)v^{(b-1)(q-1)-1}z^{\frac{p-q}2}[E-(p-1)A]Dv\\&=\frac{A-1}{A^2}(b-1)(p-q)^2\frac{\nabla v}v,\end{aligned}$$
$$\begin{aligned}\nabla_z\biggl(\frac EA\biggr)=-\frac1{A^2}|b|^{p-q}\frac{p-q}2v^{(b-1)(q-1)}z^{\frac{(p-q)}2}[E-(p-1)A]Dv=\frac{A-1}{A^2}\frac{(p-q)^2}2\frac{\nabla z}z,\end{aligned}$$
and analogously
$$\nabla_v\biggl(\frac DA\biggr)=\frac{A-1}{A^2}(b-1)(p-q)^2\frac{\nabla v}v,\qquad 
\nabla_z\biggl(\frac DA\biggr)=\frac{A-1}{A^2}\frac{(p-q)^2}2\frac{\nabla z}z.$$
Therefore,
$$\begin{aligned}\nabla \Delta v=&
\frac {1}{b|b|^{q-2}}\frac{\nabla\bigl(\mathcal Bu\bigr)}{A}
\frac{z^{1-q/2}}{v^{(b-1)(q-1)}}-\frac{q-2}2\frac {1}{b|b|^{q-2}}\frac{\mathcal Bu}{A}\frac{z^{-q/2}}{v^{(b-1)(q-1)}}\nabla z
\\&
-(p-q)\frac {|b|^{p-2q+2}}b\frac{\mathcal Bu}{A^2}v^{(b-1)(p-2q+1)}z^{1-q+p/2}\biggl[(b-1)\frac{\nabla v}v +\frac1{2}\frac{\nabla z}z\biggr]
\\&
-\frac {(b-1)(q-1)}{b|b|^{q-2}}\frac{\mathcal Bu}{A}
\frac{z^{1-q/2}}{v^{(b-1)(q-1)+1}}\nabla v
\\&-(b-1)^2|b|^{p-q}(p-q)^2\frac 1{A^2}v^{(b-1)(p-q)-2}z^{1+(p-q)/2}\,
\nabla v\\&-\frac12(b-1)|b|^{p-q}(p-q)^2\frac 1{A^2}\,v^{(b-1)(p-q)-1}z^{(p-q)/2} 
\, \nabla z\\&-
(b-1)\frac EA\biggl(\frac{\nabla z}v-z\frac{\nabla v}{v^2}\biggr)-\frac12 \frac DA\frac{\nabla\bigl(\langle\nabla z,\nabla v\rangle\bigr)}{z}+\frac12 \frac DA\frac{\langle\nabla z,\nabla v\rangle}{z^2}\nabla z
 \\&-\frac 12(b-1)(p-q)^2\frac{A-1}{A^2}\frac{\langle \nabla z,\nabla v\rangle}{zv}\nabla v
\\& -\frac 14\frac {A-1}{A^2}(p-q)^2\frac{\bigl<\nabla z,\nabla v\bigr>}{z^2}\nabla z.
\end{aligned}$$
This yields


$$\begin{aligned}\nabla \Delta v=&
\frac {1}{b|b|^{q-2}}\frac{\nabla\bigl(\mathcal Bu\bigr)}{A}
\frac{z^{1-q/2}}{v^{(b-1)(q-1)}}-\frac{q-2}2\frac {1}{b|b|^{q-2}}\frac{\mathcal Bu}{A}\frac{z^{-q/2}}{v^{(b-1)(q-1)}}\nabla z
\\&-\frac {(p-q)(b-1)}{b|b|^{q-2}}\frac{A-1}{A^2}\mathcal Bu \, \frac{z^{\frac{2-q}2}}{v^{(b-1)(q-1)+1}}\nabla v
-\frac 12\frac {p-q}{b|b|^{q-2}}\frac{A-1}{A^2}\mathcal Bu\frac{z^{-\frac q2}}{v^{(b-1)(q-1)}} \nabla z
\\&
-\frac {(b-1)(q-1)}{b|b|^{q-2}}\frac{\mathcal Bu}{A}
\frac{z^{\frac{2-q}2}}{v^{(b-1)(q-1)+1}} \nabla v
\\&
-(b-1)^2(p-q)^2\frac {A-1}{A^2}
\frac{z}{v^2}\nabla v-\frac12(b-1)(p-q)^2\frac {A-1}{A^2}\,\frac {\nabla z}v\\&
-(b-1)\frac EA\biggl(\frac{\nabla z}v-z\frac{\nabla v}{v^2}\biggr)-\frac12 \frac DA\frac{\nabla\bigl(\langle\nabla z,\nabla v\rangle\bigr)}{z}+\frac12 \frac DA\frac{\langle\nabla z,\nabla v\rangle}{z^2}\nabla z
\\&-\frac 12(b-1)(p-q)^2\frac{A-1}{A^2}\frac{\langle \nabla z,\nabla v\rangle}{zv}\nabla v
\\&-\frac 14\frac {A-1}{A^2}(p-q)^2\frac{\bigl<\nabla z,\nabla v\bigr>}{z^2}\nabla z.
\end{aligned}$$
Consequently, using $z=|\nabla v|^2$,
$$\begin{aligned}\langle\nabla \Delta v,\nabla v\rangle&=\frac {1}{b|b|^{q-2}}\frac{\langle\nabla\bigl(\mathcal Bu\bigr),\nabla v\rangle}{A}
\frac{z^{1-q/2}}{v^{(b-1)(q-1)}}-\frac{q-2}2\frac {1}{b|b|^{q-2}}\frac{\mathcal Bu}{A}\frac{z^{-q/2}}{v^{(b-1)(q-1)}}\langle\nabla z,\nabla v\rangle
\\&-\frac {(b-1)(p-q)}{b|b|^{q-2}}\frac{A-1}A\frac{\mathcal Bu}{A}
\frac{z^{\frac{4-q}2}}{v^{(b-1)(q-1)+1}}
-\frac {(b-1)(q-1)}{b|b|^{q-2}}\frac{\mathcal Bu}{A}
\frac{z^{\frac{4-q}2}}{v^{(b-1)(q-1)+1}} 
\\&-\frac 12\frac {p-q}{b|b|^{q-2}}\frac{A-1}{A^2}\mathcal Bu\frac{z^{-\frac q2}}{v^{(b-1)(q-1)}} \langle\nabla z,\nabla v\rangle -(b-1)\biggl[(b-1)\frac{(p-q)^2}{A^2}(A-1)-\frac EA\biggr]\frac{z^2}{v^2}
\\&-(b-1)\biggl[\frac{(p-q)^2}{ 2 A^2}(A-1)+\frac EA\biggr]\frac{\bigl<\nabla z,\nabla v\bigr>}{v}
\\&+\frac12\biggl[-\frac 12\frac{(p-q)^2}{A^2}(A-1)+\frac DA\biggr]\frac{\langle\nabla z,\nabla v\rangle^2}{z^2}-\frac12 \frac DA\frac{\langle\nabla\langle\nabla z,\nabla v\rangle,\nabla v\rangle}{z}.
\end{aligned}$$
On the other hand,
$$\begin{aligned}\bigl(\Delta v\bigr)^2=&
 \frac {1}{b^{2(q-1)}}\frac{\bigl(\mathcal Bu\bigr)^2}{A^2}
\frac{z^{2-q}}{v^{2(b-1)(q-1)}}+
(b-1)^2\frac {E^2z^2}{A^2v^2}+\frac14 \frac {D^2}{A^2}\frac{\langle\nabla z,\nabla v\rangle^2}{z^2}
\\&-2E\frac {b-1}{b|b|^{q-2}}\frac{\mathcal Bu}{A^2}\frac{z^{1-q/2}}{v^{(b-1)(q-1)}}\frac{z}{v}+\frac {DE}{A^2}(b-1)\frac{\langle\nabla z,\nabla v\rangle}v\\&-
\frac {D}{A^2}\frac {1}{b|b|^{q-2}}\mathcal Bu \frac{z^{-q/2}}{v^{(b-1)(q-1)}}\langle\nabla z,\nabla v\rangle.
\end{aligned}$$
Using  B$\ddot{\rm o}$chner formula, we have
	$$ \frac{1}{2}\Delta z \geq\frac{1}{N}(\Delta v)^2+\left<\nabla \Delta v, \nabla v\right>,
$$
and by 
$$\langle\nabla\langle\nabla z,\nabla v\rangle,\nabla v\rangle=
\left<D^2z\nabla v,\nabla v\right>+\frac12|\nabla z|^2,$$
since $|b|^{p-q}v^{(b-1)(p-q)}z^{(p-q)/2}=A-1$, we have

$$\begin{aligned} \frac12\Delta z& -\frac {1}{Nb^{2(q-1)}}\frac{\bigl(\mathcal Bu\bigr)^2}{A^2}
\frac1{v^{2(b-1)(q-1)}z^{q-2}}
\\&\ge
-\frac{1}{b|b|^{q-2}}\frac{\mathcal Bu}{A}\frac{z^{-q/2}}{v^{(b-1)(q-1)}}\biggl[\frac1N\frac D{A}+\frac{q-2}2+\frac{p-q}2\frac{A-1}A\biggr]\langle\nabla z,\nabla v\rangle
\\& \quad
-\frac{b-1}{A^2}\biggl[-\frac{DE}N+{ \frac12}(p-q)^2(A-1)+EA\biggr]\frac{\langle\nabla z,\nabla v\rangle}v
\\&\quad +\frac 12\biggl[{\cb -}\frac12\frac{(p-q)^2}{A^2}(A-1)+\frac DA+\frac1{2N} \frac {D^2}{A^2}\biggr]\frac{\langle\nabla z,\nabla v\rangle^2}{z^2}
\\&\quad-\frac{b-1}{b|b|^{q-2}}\frac1{v^{(b-1)(q-1)}z^{(q-2)/2}}\biggl[ \frac{2}{N} \frac{E}{A}+q-1+(p-q)\frac{A-1}A\biggr]\frac{\mathcal Bu}A\frac zv
\\&\quad -\frac{b-1}{A^2}\biggl[-\frac{b-1}NE^2+(b-1)(p-q)^2(A-1)- EA\biggr]\frac{z^2}{v^2} \\&\quad +\frac {1}{b|b|^{q-2}}\frac{\langle\nabla\bigl(\mathcal Bu\bigr),\nabla v\rangle}{A}
\frac{1}{v^{(b-1)(q-1)}z^{(q-2)/2}}
-\frac14\frac DA\frac{|\nabla z|^2}{z}
-\frac12\frac DA\frac{\langle D^2z\nabla v,\nabla v\rangle}{|\nabla v|^2}.\end{aligned}$$
Now, considering the definition of the operator   
$$\mathscr{A}_v(z):=-\Delta z-\frac DA\frac{\langle D^2z\nabla v,\nabla v\rangle}{|\nabla v|^2},$$
and using $D+p-q=  (p-2)A$ and $E+p-q=  (p-1)A$,
we obtain the required inequality being
$$\frac1N\frac D{A}+\frac{q-2}2+\frac{p-q}2\frac{A-1}A=\biggl(\frac 1N+\frac 12\biggr)\frac DA$$
and
$$\frac2N\frac {E}{A}+q-1+(p-q)\frac{A-1}A=2\biggl(\frac 1N+\frac 12\biggr)\frac EA.$$
\end{proof}

\section{The Hamilton Jacobi type case}
In this section we deal with equation \eqref{main_HJ}.
In this particular case, for any  solution $u$  of \eqref{main_HJ}, the change of variable $v=u^{1/b}$ will not be used, that is we consider $b=1$, and because of this we also do not require solution to be nonnegative. Therefore, taking $z=|\nabla u|^2$,  the functions $A,D,E$ become
$$A:=z^{(p-q)/2}+1, \qquad
D:=(p-2)z^{(p-q)/2}+q-2,$$
$$E:=(p-1)z^{(p-q)/2}+q-1.$$
Furthermore, the operator 
$$\mathscr{A}_u(z)=-\Delta z-\frac DA\frac{\left<D^2z\nabla u,\nabla u\right>}{|\nabla u|^2},$$
by Lemma \ref{lemma_Bcal}, satisfies   the following inequality
\begin{equation}\label{ineq_b=1}\begin{aligned} \frac12\mathscr{A}_v(z)& +\frac 1{N}\frac{\bigl(\mathcal Bu\bigr)^2}{A^2}
z^{2-q}
\\
&\le 
\frac{\mathcal Bu}{A}z^{-q/2}\biggl(\frac1N+\frac12 \biggr)\frac D{A}\langle\nabla z,\nabla u\rangle
-\frac{\langle\nabla\bigl(\mathcal Bu\bigr),\nabla u\rangle}{A}
z^{(2-q)/2} 
\\
&-\frac 1{2}\biggl[-\frac12\frac{(p-q)^2}A\frac{A-1}A+\frac DA+\frac1{2N}\frac{D^2}{A^2}\biggr]\frac{\langle\nabla z,\nabla u\rangle^2}{z^2}+\frac14\frac DA\frac{|\nabla z|^2}{z}\quad \mbox{on}\quad \{z>0\},
\end{aligned}\end{equation}
for any nonnegative solution $u$  of \eqref{main_HJ}.

\medskip{\it {\bf Proof of Theorem \ref{th_HJ}}.} We first replace $\mathcal Bu=-|\nabla u|^{m}=-z^{m/2}$ in \eqref{ineq_b=1}, yielding
$$\begin{aligned} \frac12\mathscr{A}_u(z) +\frac 1{N}\frac{z^{m-q+2}}{A^2}&\le
-\frac{z^{(m-q)/2}}{A}\biggl(\frac1N+\frac12 \biggr)\frac D{A}\langle\nabla z,\nabla u\rangle
\\&\quad +\frac m2\frac{z^{(m-q)/2}}{A}\langle\nabla z,\nabla u\rangle
+ C_2\frac{|\nabla z|^2}{z},
\end{aligned}$$
where, by \eqref{estim_DA_EA} and being $z=|\nabla u|^2$,  we have used
\begin{equation}\label{ineq_C_1}\begin{aligned}\biggl|\frac 1{2}\biggl[-\frac12\frac{(p-q)^2}A\frac{A-1}A+\frac DA+\frac1{2N}\frac{D^2}{A^2}\biggr]\frac{\langle\nabla z,\nabla v\rangle^2}{z^2}\biggr| 
\le C_1\frac{|\nabla z|^2}z,
\end{aligned}\end{equation}
with\begin{equation}\label{C_1}C_1=\frac14(p-q)^2+\frac{|p-2|}2\biggl(1+\frac{p-2}{2N}\biggr),\qquad C_2=\frac 14|p-2|+C_1.\end{equation}
Furthermore,  by \eqref{estim_DA_EA}, estimating as follows
$$\begin{aligned}\frac{z^{(m-q)/2}}{A}\biggl|\biggl(\frac1N+\frac12 \biggr)\frac D{A}-\frac m2\biggr||\langle\nabla z,\nabla u\rangle|
&\le 
\frac{1}{A}\biggl|\biggl(\frac1N+\frac12 \biggr)\frac D{A}-\frac m2\biggr|\frac{|\nabla z||\nabla u|}zz^{(m-q+2)/2}
\\&\le \varepsilon\frac1{A^2}z^{m-q+2}+C_\varepsilon \frac{|\nabla z|^2}{z},
\end{aligned}$$
we reach
$$\frac12\mathscr{A}_u(z) +\biggl(\frac 1{N}-\varepsilon\biggr )\frac1{A^2}z^{m-q+2}\le  C\frac{|\nabla z|^2}{z}. $$
Equivalently, for $\varepsilon>0$ small enough and replacing the expression of $A$, it follows
\begin{equation}\label{ineq_HJ}\mathscr{A}_u(z) +a\frac {z^{m-q+2}}{[z^{(p-q)/2}+1]^2}\le  C\frac{|\nabla z|^2}{z},  \quad\mbox{on}\, \{z>0\}, \end{equation}
where $a,C$ are positive constants depending on $p,q,N,m$.
Now, we complete the proof.

\medskip{(i)} If $z>1$, then $z^{(p-q)/2}>1$, hence
$$\frac1{(2z^{(p-q)/2})^2}<\frac 1{[z^{(p-q)/2}+1]^2}.$$
This yields from \eqref{ineq_HJ} that
\begin{equation}\label{ineq_HJ2}
\mathscr{A}_u(z) +\tilde a z^{m-p+2}\le  C\frac{|\nabla z|^2}{z},  \quad\mbox{on }\, \{z>1\}.
\end{equation}
Now, setting $\tilde z=z-1$, \eqref{ineq_HJ2} reduces to
$$
\mathscr{A}_u(\tilde z) +\tilde a (\tilde z+1)^{m-p+2}\le  C\frac{|\nabla \tilde z|^2}{\tilde z+1} \quad\mbox{on }\, \{\tilde z>0\}, 
$$
which leads to
$$\mathscr{A}_u(\tilde z) +\tilde a \tilde z^{m-p+2}\le  C\frac{|\nabla \tilde z|^2}{\tilde z} \quad\mbox{on }\, \{\tilde z>0\}. $$
By Lemma \ref{lem1}, we obtain
$$\tilde z(x)\leq C\left({\rm dist}\left(x,\partial\Omega\right)\right)^{-\frac2{m-p+1}}.$$
Hence \eqref{1-8} follows immediately by replacing $\tilde z=z-1$.

(ii) If $z>\varepsilon^\frac{2}{p-q}$, then $z^\frac{p-q}{2}> \varepsilon$, and hence
$$\frac1{(1+\varepsilon^{-1})^2z^{p-q}}\leq\frac 1{[z^{(p-q)/2}+1]^2}.$$
Therefore from \eqref{ineq_HJ}, we have
\begin{equation}\label{ineq_HJ2i}\mathscr{A}_u(z) +\frac{a}{(1+\varepsilon^{-1})^2} z^{m-p+2}\le  C\frac{|\nabla z|^2}{z},  \quad\mbox{on}\, \{z\geq\varepsilon^\frac{2}{p-q}\}. \end{equation}
Set $\tilde z=z-\varepsilon^\frac{2}{p-q}$ as before to yield
$$\mathscr{A}_u(\tilde z) +\frac{a}{(1+\varepsilon^{-1})^2} \tilde z^{m-p+2}\le  C\frac{|\nabla \tilde z|^2}{\tilde z},  \quad\mbox{on}\, \{\tilde z>0\}.$$ 
Therefore, applying \cite[Lemma 3.1]{BV}(with $\beta=0$), we obtain
$$\tilde z(x)\leq C (1+\varepsilon^{-1})^\frac{2}{m-p+1}\left({\rm dist}\left(x,\partial\Omega\right)\right)^{-\frac2{m-p+1}},$$
i.e.,
$$z(x)\leq \varepsilon^\frac{2}{p-q}+C (1+\varepsilon^{-1})^\frac{2}{m-p+1}\left({\rm dist}\left(x,\partial\Omega\right)\right)^{-\frac2{m-p+1}}.$$
Hence, if $\Omega=\mathbb{R}^N$, the above inequality reduces to
$$z(x)\leq \varepsilon^\frac{2}{p-q} \quad\forall\, \varepsilon>0.$$ Hence, taking $\varepsilon\to 0$ we get  $z=0$, i.e., $u$ is constant. 
\hfill$\square$

\section{Proof of Theorem \ref{th_sm}}

This section is devoted to the solutions of equation \eqref{main_sm}.

{\it {\bf Proof of Theorem \ref{th_sm}}. } Let $u$ be a solution of \eqref{main_sm}. Differently from the Hamilton Jacobi type case, here we need to consider the change of variables $u=v^b$ so that  the inequality in the statement of  Lemma \ref{lemma_Bcal}, when
$$\mathcal Bu=-u^s|\nabla u|^m=-|b|^mv^{bs+m(b-1)}z^{m/2},$$
so that 
$$\langle\nabla\bigl(\mathcal Bu\bigr),\nabla v\rangle=-|b|^m
v^{(b-1)m+bs-1}z^{1+m/2}\bigl[bs+m(b-1)+\frac m2\frac{v}{z^2}\langle\nabla z,\nabla v\rangle\bigr],$$
gives 
$$\begin{aligned} \frac12\mathscr{A}_v(z)& +\frac {|b|^{2m}}{Nb^{2(q-1)}}\frac1{A^2}v^{2t}z^{m-q+2}
\\&\le 
-\frac{|b|^m}{b|b|^{q-2}}\frac{1}{A}v^t{z^{(m-q)/2}}\biggl(\frac 1N+\frac 12\biggr)\frac DA\langle\nabla z,\nabla v\rangle
\\
& \quad-2\frac{(b-1)|b|^m}{b|b|^{q-2}}\frac{1}Av^{t}z^{(m-q+2)/2}\biggl(\frac 1N+\frac 12\biggr)\frac EA\frac zv
\\&\quad +\frac {bs+m(b-1)}{b|b|^{q-m-2}}\frac{v^{t-1}z^{(m-q+4)/2}}{A}
 +\frac m{2b}\frac{|b|^{m-q+2}}{A}v^tz^{(m-q+2)/2}\frac{\langle\nabla z,\nabla v\rangle}z
\\&\quad+(b-1)\biggl[-\frac 1N\frac{DE}{A^2}+{ \frac12}\frac{(p-q)^2}A\frac{A-1}A+\frac EA\biggr]\frac{\langle\nabla z,\nabla v\rangle}v
\\&\quad-(b-1)\biggl[\frac{b-1}N\frac{E^2}{A^2}-(b-1)\frac{(p-q)^2}A\frac{A-1}A+ \frac EA\biggr]\frac{z^2}{v^2}+C_2\frac{|\nabla z|^2}{z}\quad \mbox{on}\quad \{z>0\},
\end{aligned}$$
where  
\begin{equation}
\label{def_t}t:=(b-1)(m-q+1) +bs
\end{equation}
and $C_2$ is given in \eqref{C_1}.
Now,  proceed with the following estimates by Young inequality and thanks to \eqref{estim_DA_EA}
\begin{equation}\label{nabla/v}\begin{aligned}
\biggl|(b-1)&\biggl[-\frac 1N\frac{DE}{A^2}+{\frac12}\frac{(p-q)^2}A\frac{A-1}A+\frac EA\biggr]\biggl|\cdot\biggl|\frac{\langle\nabla z,\nabla v\rangle}v\biggr|
    \\&\le 
|b-1|\biggl|\frac 1N (p-2)(p-1) +(p-q)^2 +p-1\biggr| \cdot \frac{|\nabla z||\nabla v|}{z}\frac zv\\&\le \varepsilon\frac {z^2}{v^2}+C_\varepsilon\frac{|\nabla z|^2|\nabla v|^2}{z^2} =\varepsilon\frac {z^2}{v^2}+ C_\varepsilon\frac{|\nabla z|^2}z,
\end{aligned}\end{equation}
and
\begin{align*}
\biggl|\frac{1}{b|b|^{q-2-m}A}&\biggl(\frac 1N+\frac 12\biggr)\frac DA\biggr|v^t{z^{(m-q)/2}}\bigl|\langle\nabla z,\nabla v\rangle\bigr|
\\
&\le \frac{1}{|b|^{q-1-m}}\frac 1A\biggl(\frac 1N+\frac 12\biggr)|p-2|v^t{z^{(m-q+2)/2}} \frac{|\nabla v|}{z}\,|\nabla z|
\\
&\le \varepsilon \frac{|b|^{2(m-q+1)} }{A^2}v^{2t}z^{m-q+2} +C_\varepsilon\frac{|\nabla z|^2}{z}, 
\end{align*}
so that the inequality for the operator $\mathscr{A}_v(z)$ becomes
\begin{equation}\label{est_last}
\begin{aligned}
\frac12 \mathscr A_v(z)&+\biggl(\frac1N-\varepsilon\biggr)\frac{|b|^{2(m-q+1)}}{A^2}v^{2t}z^{m-q+2}
+(A_1-\varepsilon)\frac{z^2}{v^2}\\&\quad+
\frac{|b|^{m-q+2}}{bA}\biggl[2(b-1)\biggl(\frac 1N+\frac 12\biggr)\frac EA-bs-m(b-1)\biggr]v^{t-1}z^{(m-q+4)/2}
\\&\le\frac m2\frac{|b|^{m-q+2}}{bA}v^{t}z^{(m-q+2)/2}\frac{\langle\nabla z,\nabla v\rangle}z
+\bigl(2 C_\varepsilon+C_2\bigr)
 \frac{|\nabla z|^2}z, \end{aligned}
\end{equation}
where
\begin{equation}\label{A_1}A_1=(b-1)\biggl[\frac{b-1}N\frac{E^2}{A^2}-(b-1)\frac{(p-q)^2}A\mathfrak A +\frac EA\biggr], \qquad  \mathfrak A:=\frac{A-1}A.\end{equation}
Finally, estimating as follows 
$$\frac m2\frac{|b|^{m-q+1}}{A}v^{t}z^{(m-q+2)/2}\biggl|\frac{\langle\nabla z,\nabla v\rangle}z\biggr|\le \varepsilon \frac{|b|^{2(m-q+1)} }{A^2} v^{2t}z^{m-q+2} +C_\varepsilon\frac{|\nabla z|^2}{z}$$
we have
\begin{equation}
\label{mathscrA_A1A3}
\frac12 \mathscr A_v(z)+\biggl(\frac1N-2\varepsilon\biggr)\frac{|b|^{2(m-q+1)}}{A^2}v^{2t}z^{m-q+2}
+(A_1-\varepsilon)\frac{z^2}{v^2}
+A_2v^{t-1}z^{(m-q+4)/2}\le C\frac{|\nabla z|^2}z,
\end{equation}
for some positive constant $C$ and with 
$$A_2=\frac{|b|^{m-q+2}}{bA}\biggl [(b-1)\frac 2N \frac EA-t+ (b-1)(p-q)\mathfrak A \biggr],$$
where we have used that 
$$\frac{(b-1)E}{A}- bs-m(b-1)=\frac{b-1}A[(p-1)A-p+q-(q-1)A]-t=\frac{b-1}A(p-q)(A-1)-t.$$
Define
$$\mathcal H:=\biggl(\frac1N-2\varepsilon\biggr)\frac{|b|^{2(m-q+1)}}{A^2}v^{2t}z^{m-q+2}
+(A_1-\varepsilon)\frac{z^2}{v^2}
 +A_2v^{t-1}z^{(m-q+4)/2}$$
Then \eqref{mathscrA_A1A3} reduces to 
\begin{equation}\label{mathscrA_A1A3-1}
\frac12 \mathscr A_v(z)+\mathcal{H}\le C\frac{|\nabla z|^2}z,
\end{equation}
 
Now we set
 \begin{equation}\label{zeta}\zeta=v^{t+1}z^{\frac{m-q}2}.\end{equation}
This in turn implies
$$\mathcal H= \bigg(\varepsilon\frac{|b|^{2(m-q+1)}}{A^2}\zeta^2 + \mathcal T_\varepsilon(\zeta)\bigg)\frac{z^2}{v^2},$$
where
\begin{equation}\label{trin}
\mathcal T_\varepsilon(\zeta)= 
\frac{|b|^{2(m-q+1)}}{A^2}\biggl(\frac 1N- 3\varepsilon\biggr)\zeta^2+
A_2\zeta+\bigl(A_1-\varepsilon\bigr).
\end{equation}
The  discriminant of the trinominal $\mathcal T_\varepsilon(.)$ is given by 
\begin{equation}\label{D}
\mathcal D:=\frac{|b|^{2(m-q+1)}}{A^2}
\biggl\{\biggl[(b-1)\frac 2N \frac EA-t+(b-1)(p-q)\mathfrak A\biggr]^2- 4\Bigl(\frac{1}{N}-3\varepsilon\Bigr) (A_1-\varepsilon)\biggr\}.
\end{equation}
In the following we will show that we can choose $b$ suitably so that for some constant $\kappa, \varepsilon>0$, we will have
\begin{equation}\label{discri}
\mathcal D\leq -\kappa \frac{|b|^{2(m-q+1)}}{A^2}.
\end{equation}
This actually would show that $\mathcal D$ is strictly negative
and
$$\mathcal T_\varepsilon(\zeta)\geq \min_{\mathbb{R}^N}\mathcal T_\varepsilon= \frac{-\mathcal{D} A^2}{4 |b|^{2(m-q+1)}\big(\frac{1}{N}-3\varepsilon\big)} \geq\frac{N\kappa}{4(1-3N\varepsilon)}.
$$
Assuming the choice of $b$ satisfying \eqref{discri}, we first complete the proof. From the above estimate, we see 
\begin{equation}\label{Sep11-2}
\mathcal{H} \frac{v^2}{z^2}\ge \varepsilon \frac{|b|^{2(m-q+1)}}{A^2}\zeta^2 + \frac{N\kappa}{4(1-3N\varepsilon)}
\geq \upkappa_1 + \upkappa_2  \frac{1}{A^2}\max\{\zeta^2, 1\}\geq \upkappa_1 + \upkappa_2  \frac{1}{A^2}\zeta^{\theta} ,
\end{equation}
for some  positive constant $\upkappa_1, \upkappa_2$ depending on $\varepsilon, \kappa$ and $b$. Last inequality holds for any $\theta\in[0,2]$.
We choose $b$ in \eqref{def_t} such that $t>0$.

As in the proof of Theorem~\ref{th_HJ}, we will consider $\{z\geq 1\}$. Now if $b\in (0, 1]$   being  $t>0$ by \eqref{b>0}, in particular
$$\frac{m-q+1}{m-q+1+s}<b\leq 1,$$
we take $\theta:=\frac{2}{t+1}\in(0,2)$.
On the set $\{z\geq 1\}$, substituting the definition of $\zeta$ from \eqref{zeta}, we estimate

\begin{align}\label{Sep7-1}
\mathcal{H} &\geq \upkappa_1 \frac{z^2}{v^2} \chi_{\{v\leq 1\}} +   \frac{\upkappa_2}{A^2}\zeta^\theta \frac{z^2}{v^2} \chi_{\{v\geq 1\}}\nonumber
\\
&\geq \upkappa_1 z^2 \chi_{\{v\leq 1\}} +   \bigg(\frac{\upkappa_2 (v^{t+1}z^{\frac{m-q}{2}})^{\frac{2}{t+1}}}{z^{p-q}(1+b^{p-q}v^{(b-1)(p-q)})^2}\bigg)
 \frac{z^2}{v^2} \chi_{\{v\geq 1\}}\nonumber
\\
&\ge \upkappa_1 z^2 \chi_{\{v\leq 1\}} +  \frac{\upkappa_2}{(1+b^{p-q})^2}z^{\frac{m-q}{t+1}+2+q-p} \chi_{\{v\geq 1\}}\nonumber
\\
&\ge \kappa_b\, z^{1+ \min\{1, \upbeta_1\}} \quad\text{on}\quad \{z\geq 1\},
\end{align}
where
$\kappa_b>0$ is a constant and
\begin{equation}\label{beta1}
\upbeta_1=(\frac{m-q}{t+1}+1+q-p)=\frac{b(m+s-q+1)}{b(m+s-q+1)+q-m}- (p-q).
\end{equation}

Next suppose $b>1$. Note that
\begin{align*}
2t&=2(b-1)(m-q+1)+2bs>2(b-1)(p-s-q) + 2bs
= 2(b-1)(p-q)+2s\\ 
&> 2(b-1)(p-q).
\end{align*}
Here we set $\theta= \frac{2(b-1)(p-q)+2}{t+1}\in (0, 2)$.
Then again on $\{z\geq 1\}$, substituting the value of $\zeta$ from \eqref{zeta}, we estimate
\begin{align}\label{Sep7-2}
\mathcal{H} &\geq \upkappa_1 \frac{z^2}{v^2} \chi_{\{v\leq 1\}} +  \frac{\upkappa_2}{A^2}\zeta^\theta \frac{z^2}{v^2} \chi_{\{v\geq 1\}}\nonumber
\\
&\geq \upkappa_1 z^2 \chi_{\{v\leq 1\}} +  \bigg(\frac{\upkappa_2 (v^{t+1}z^{\frac{m-q}{2}})^{\theta} }{z^{p-q}(1+b^{p-q}v^{(b-1)(p-q)})^2}\bigg)
\frac{z^2}{v^2} \chi_{\{v\geq 1\}}\nonumber
\\
&\geq \upkappa_1 z^2 \chi_{\{v\leq 1\}}  +   \bigg(\frac{\upkappa_2 v^{2(b-1)(p-q)}}{(1+b^{p-q}v^{(b-1)(p-q)})^2}\bigg)z^{\frac{\theta(m-q)}{2} +2+q-p}  \chi_{\{v\geq 1\}}\nonumber
\\
&\geq \upkappa_1 z^2 \chi_{\{v\leq 1\}} +  \upkappa_2
\left[\inf_{\xi\geq 1} \frac{ \xi^2}{(1+b^{p-q}\xi)^2} \right]z^{\frac{\theta(m-q)}{2} +2+q-p}  \chi_{\{v\geq 1\}}\nonumber
\\
&\ge \upkappa_1 z^2 \chi_{\{v\leq 1\}} +  \frac{\upkappa_2 }{(1+b^{p-q})^2}z^{\frac{\theta(m-q)}{2} +2+q-p}  \chi_{\{v\geq 1\}}\nonumber
\\
&\ge \kappa_b\, z^{1+ \min\{1,\upbeta_2\}} \quad\text{on} \quad \{z\geq 1\},
\end{align}
where 
\begin{align}\label{beta2}
\upbeta_2&= \frac{(b-1)(p-q)+1}{t+1}(m-q) + 1-(p-q)\nonumber
\\
&= \frac{(b-1)(p-q)(m-q)}{b(m+s-q+1)+q-m}+\frac{b(m+s-q+1)}{b(m+s-q+1)+q-m}- (p-q)\nonumber
\\
&=\frac{(b-1)(p-q)(m-q)}{b(m+s-q+1)+q-m}+\upbeta_1.
\end{align}
 We observe that 
$$\lim_{b\to\infty}\upbeta_2=\frac{(p-q)(m-q)}{m+s-q+1}+1- (p-q)=1-\frac{(p-q)(1+s)}{m+s-q+1}.$$

Once we prove 
\begin{equation}\label{gamma}
\gamma:=\min\{1,\chi_{(0, 1]}(b) \upbeta_1 \} + \min\{1, \chi_{(1, \infty)}(b)\upbeta_2 \}
\end{equation}
is positive,
inserting \eqref{Sep7-1} and \eqref{Sep7-2} into \eqref{mathscrA_A1A3-1} will lead to 
\begin{equation}\label{Sep11-1}
\frac12 \mathscr A_v(z)+\kappa z^{1+\gamma}\le C\frac{|\nabla z|^2}z\quad\text{on} \quad \{z\geq 1\}.
\end{equation}
Hence by Lemma~\ref{lem1}  and employing an argument similar to Theorem~\ref{th_HJ} we obtain
\begin{equation}\label{z-1}z(x)\leq \left(1+C\left({\rm dist}\left(x,\partial\Omega\right)\right)^{-\frac{2}{\gamma}}\right).\end{equation}

\smallskip

 Now it remains to prove \eqref{discri} holds and $\gamma$ is positive. To this aim, 
from \eqref{D}, we first write
$$
\mathcal{D}= \frac{|b|^{2(m-q+1)}}{A^2}\mathcal{L},
$$
where 
$$\mathcal{L}:=
\biggl\{\biggl[(b-1)\frac 2N \frac EA-t+(b-1)(p-q)\mathfrak A\biggr]^2-4\Bigl(\frac{1}{N}-3\varepsilon\Bigr) (A_1-\varepsilon),$$

From the definition of $t$ in \eqref{def_t} it follows that
\begin{equation}\label{Q}
b-1=\frac{t-s}{m+s-q+1}:=\frac{t-s}{\mathcal Q}
\end{equation}
Therefore, 
$$\mathcal L=\biggl[t\biggl(\frac \Gamma{\mathcal Q}-1\biggr)-s\frac \Gamma{\mathcal Q}\biggr]^2-4\Bigl(\frac{1}{N}-3\varepsilon\Bigr) (A_1-\varepsilon),\quad\text{where} \quad \Gamma:=\frac{2}{N}\frac{E}{A}+(p-q)\mathfrak A, $$
and from \eqref{A_1}, it follows
$$A_1=\frac{t^2}{\mathcal Q^2}\Xi +\frac t{\mathcal Q}\biggl(-2\frac s{\mathcal Q}\Xi+\frac EA\biggr)+
\frac s{\mathcal Q}\biggl(\frac s{\mathcal Q}\Xi-\frac EA\biggr),\quad\text{where}\quad\Xi:=\frac{1}{N}\frac{E^2}{A^2}-\frac 1A(p-q)^2\mathfrak A.$$
Consequently
$$\mathcal L=L_1t^2+L_2t+L_3,$$
with
$$\begin{aligned}
L_1&:=1-\frac2{\mathcal Q}\Gamma+\frac{\Upsilon}{
 \mathcal{Q}^2}
+\varepsilon\frac{12E^2}{NA^2\mathcal{Q}^2}
-\varepsilon\frac{12}{A\mathcal{Q}^2}(p-q)^2\mathfrak{A}
\\
L_2&:=\frac{2s}{\mathcal Q}\Gamma-\frac{2s}{\mathcal Q^2}\Upsilon-
\frac{4}{N}\frac{E}{A}\frac1{\mathcal Q}+\varepsilon\frac{12}{\mathcal Q}\bigg(\frac{E}{A}-\frac{2s}{\mathcal Q}\frac{E^2}{NA^2}+\frac{2s}{{\mathcal Q}A}(p-q)^2\mathfrak{A}\bigg)
\\
L_3&:=\frac{s^2}{\mathcal Q^2}\Upsilon+\frac{4}{N}\frac{E}{A}\frac s{\mathcal Q}+4\varepsilon\bigg(\frac{3s}{\mathcal{Q}}\bigl(\frac s{\mathcal Q}\Xi-\frac EA\bigr)+(\frac{1}{N}-3\varepsilon) \bigg),\end{aligned}$$
where
$$ \Upsilon:=(p-q)^2\mathfrak A^2+\frac4N(p-1)(p-q)\mathfrak A.$$
In particular, it holds
\begin{equation}\label{estimate_A_Y}
\begin{gathered}
0\le \mathfrak A\le 1, \qquad 0\le \Upsilon\le (p-q)\biggl[p-q+\frac 4N (p-1)\biggr] \\
\frac2N (q-1)\le \Gamma\le \frac2N (p-1) +p-q,\qquad
\Xi\leq \frac{(p-1)^2}{N}.
\end{gathered}
\end{equation}
It is important to note that the coefficients of the polynomial $\mathcal{L}$, may depend on $v, z$ due to the involvement of $E, A, \mathfrak{A}$. We would like to define a polynomial $\tilde{\mathcal{L}}$ with deterministic coefficients
that dominates $\mathcal{L}$ in $[0, \infty)$. To do so,  we let
\begin{align*}
\tilde{L}_1&:=1-\frac{4(q-1)}{N\mathcal Q}+ \frac{(p-q)}{\mathcal{Q}^2}\biggl[p-q+\frac 4N (p-1)\biggr]
+\varepsilon\frac{12(p-1)^2}{N \mathcal{Q}^2},
\\
\tilde{L}_2&:=\frac{2s}{\mathcal Q}\left(\frac2N (p-1) +p-q\right)-
\frac{4(q-1)}{N\mathcal Q}+\varepsilon\frac{12}{\mathcal Q}\bigg(p-1-\frac{2s}{\mathcal Q}\frac{(q-1)^2}{N}+\frac{2s}{{\mathcal Q}}(p-q)^2\bigg),
\\
\tilde{L}_3&:=\frac{s^2}{\mathcal Q^2}(p-q)\biggl[p-q+\frac 4N (p-1)\biggr]+\frac{4(p-1)}{N}\frac s{\mathcal Q}+4\varepsilon\bigg(\frac{3s}{\mathcal{Q}}\bigl(\frac s{\mathcal Q}
\frac{(p-1)^2}{N}-(q-1)\bigr)+(\frac{1}{N}-3\varepsilon) \bigg).
\end{align*}
Clearly, for $\tilde{\mathcal L}(t):=\tilde{L}_1t^2+\tilde{L}_2t+\tilde{L}_3$,  by \eqref{estim_DA_EA} and \eqref{estimate_A_Y}, we have ${\mathcal L}(t)\leq \tilde{\mathcal L}(t)$
for $t\in [0, \infty)$, uniformly in $v$ and $z$. Furthermore, the choice of $t$ also determines the choice of $b$. Therefore, to establish  \eqref{discri} it is enough to find $t, \kappa, \varepsilon>0$, under the stated conditions of Theorem~\ref{th_sm},  satisfying
\begin{equation}\label{moddiscri}
\tilde{\mathcal L}(t)\leq -\kappa.
\end{equation}
Because of continuity, it is enough to establish \eqref{moddiscri}
with $\varepsilon=0$.  In this case, coefficients of $\tilde{\mathcal L}(t)$ simplify as follows

\begin{equation}\label{tildeL1}\begin{aligned}
\tilde{L}_1&:=1-\frac{4(q-1)}{N\mathcal Q}+ \frac{(p-q)}{\mathcal{Q}^2}\biggl[p-q+\frac 4N (p-1)\biggr]=\frac{1}{\mathcal{Q}^2}(\mathcal Q-\mathcal Q_1)(\mathcal Q-\mathcal Q_2),\\
\tilde{L}_2&:=\frac{2s}{\mathcal Q}\left(\frac2N (p-1) +p-q\right)-
\frac{4(q-1)}{N\mathcal Q}\\
\tilde{L}_3&:=\frac{s^2}{\mathcal Q^2}(p-q)\biggl[p-q+\frac 4N (p-1)\biggr]+\frac{4(p-1)}{N}\frac s{\mathcal Q}=\frac{s^2}{\mathcal Q^2}\mathcal{R}+\frac{4(p-1)}{N}\frac s{\mathcal Q}.
\end{aligned}\end{equation}
where $\mathcal{Q}_1$ and $\mathcal{Q}_2$ are defined in  \eqref{defQ2} and \eqref{defQ1} respectively.

\smallskip

{\bf Case 1:} 
 $\mathcal{Q}_1<\mathcal{Q}<\mathcal{Q}_2$

This immediately implies $\tilde L_1< 0$. Therefore, we can choose 
$t$ large enough so that \eqref{moddiscri} holds with $\kappa=1$. 
Moreover, since $t\to\infty$ implies $b\to\infty$  and 
$\lim_{b\to\infty}\upbeta_2=1-\frac{(p-q)(1+s)}{\mathcal Q}>0$ , by the given hypothesis.  
 Hence  for $t$  large enough it holds $\upbeta_2>0$. This proves $\gamma>0$ in \eqref{gamma}.

\smallskip

{\bf Case 2: } 
 $\mathcal Q\subseteq\{\mathcal Q_1,\mathcal Q_2\} \quad s<\dfrac{q-1}{p-1+\frac N2(p-q)}.$

Therefore, in this case we have $\tilde L_1=0$ and 
$\tilde{L}_2<0$. Which in turn implies  $\tilde{\mathcal L}(t)\to-\infty$ as $t\to \infty$, we can argue as before to find $b$ (equivalently, $t$) satisfying \eqref{b>0},  \eqref{moddiscri} with $\kappa=1$ and $\upbeta_2>0$. Hence,  $\gamma>0$ in \eqref{gamma}.

\smallskip

{\bf Case 3:}  

\begin{equation}\label{Sep8-3}
\{\mathcal Q<\mathcal Q_1\}\cup\{\mathcal Q>\mathcal Q_2\}, \quad  s<\dfrac{q-1}{p-1+\frac N2(p-q)}, \quad\mbox{and}
\quad m\leq q<p<m+1.
\end{equation}

As $\mathcal Q_1<\mathcal Q_2$, clearly in this case we have 
$\tilde L_1>0$,  and therefore, $\tilde{\mathcal L}$ forms a strictly convex function that attends minimum at the point 
$$t^*=-\frac{\tilde{L}_2}{2\tilde{L}_1}.$$
Further, \eqref{Sep8-3} also implies $\tilde{L}_2<0$. In particular, we have
\begin{equation}\label{Sep8-4}
t^*>0\implies b^*=\frac{t^*+m-q+1}{\mathcal Q}>\frac{m-q+1}{\mathcal Q}>0.
\end{equation}

Hence \eqref{b>0} holds. To establish \eqref{moddiscri} we set $\kappa:=-\tilde{\mathcal L}(t^*)$. So, to show \eqref{moddiscri} holds,  we need to prove that $\tilde{\mathcal L}(t^*)<0$. 
Since  
$\tilde{\mathcal L}(t^*)=\frac{4\tilde L_3\tilde L_1-\tilde L_2^2}{4\tilde L_1}$ and $\tilde L_1>0$, it is enough to verify that $4\tilde{L}_1\tilde{L}_3<\tilde{L}_2^2$.
 From \eqref{tildeL1} we see that 
$$\tilde L_2^2=\frac{4}{\mathcal Q^2}\biggl[\frac2N(q-1)-s\left(\frac2N (p-1) +p-q\right)\biggr]^2$$
giving us
$$\frac{\tilde{L}_2^2}{4\tilde{L}_3}=\frac{N}{s}\frac{\bigl[\frac2N(q-1)-s\left(\frac2N (p-1) +p-q\right)\bigr]^2}{Ns\mathcal R+4(p-1)\mathcal Q}. $$

Now, if $\mathcal{Q}<\mathcal{Q}_1$,  then we have
$$\frac{\tilde{L}_2^2}{4\tilde{L}_3}\geq \frac{N}{s}\frac{\bigl[\frac2N(q-1)-s\left(\frac2N (p-1) +p-q\right)\bigr]^2}{Ns\mathcal R+4(p-1)\mathcal Q_1}:=a.$$
Now 
$a> \tilde L_1,$ if and only if
\begin{equation}\label{Sep14-1}
(1-a)\mathcal Q^2-\frac{4(q-1)}{N}\mathcal Q+\mathcal R< 0.
\end{equation}

Now if $a\leq 1$ then $$\mathcal Q<\mathcal Q_1\implies (1-a)\mathcal Q^2-\frac{4(q-1)}{N}\mathcal Q+\mathcal R\leq (1-a)\mathcal Q^2_1+\mathcal R- \frac{4(q-1)}{N}\mathcal Q< 0$$
provided $Q>\frac{N\big((1-a)\mathcal Q_1^2+\mathcal R\big)}{4(q-1)}$. 

If $a>1$ then \eqref{Sep14-1} is equivalent to 
$$(a-1)\mathcal Q^2+\frac{4(q-1)}{N}\mathcal Q-\mathcal R> 0.$$
Now if $\mathcal Q>\frac{\mathcal R N}{4(q-1)}$ then
$$ (a-1)\mathcal Q^2+\frac{4(q-1)}{N}\mathcal Q-\mathcal R\geq \frac{4(q-1)}{N}\mathcal Q-\mathcal R> 0.$$
Therefore, 
$$\mathcal{Q}\in \begin{cases}\bigg(\frac{N\big((1-a)\mathcal Q_1^2+\mathcal R\big)}{4(q-1)}, \, \mathcal Q_1\bigg) \quad\mbox{if}\quad a\leq 1\\
\bigg(\frac{\mathcal R N}{4(q-1)},\, \mathcal Q_1\bigg) \quad\mbox{if}\quad a>1
\end{cases}
$$
yields \eqref{moddiscri}.

Now suppose $\mathcal Q_2<\mathcal Q<\mathcal Q_3,$ where
$$\mathcal Q_3:=\mathcal Q_2+\frac{[
\frac2N(q-1)-s\left(\frac2N (p-1) +p-q\right)]^2}{s\bigl[\frac{s\mathcal R}{\mathcal Q_2}+\frac4N(p-1)\bigr]}.$$
In this case we have
$$\begin{aligned}\frac{\tilde{L}_2^2}{4\tilde{L}_3}&=\frac{\bigg[\frac{2}{N}\big(q-1-s(p-1)\big)-s(p-q)\bigg]^2}{s\mathcal Q\bigl[\frac{s\mathcal R}{\mathcal Q}+\frac4N(p-1)\bigr]}
\\&\ge \frac{\bigg[\frac{2}{N}\big(q-1-s(p-1)\big)-s(p-q)\bigg]^2}{s\mathcal Q\bigl[\frac{s\mathcal R}{\mathcal Q_2}+\frac4N(p-1)\bigr]}
\\&=\frac{\mathcal Q_3-\mathcal Q_2}{\mathcal Q}>
\frac{\mathcal Q-\mathcal Q_2}{\mathcal Q}\ge
\frac{\mathcal Q-\mathcal Q_2}{\mathcal Q}\cdot\frac{\mathcal Q-\mathcal Q_1}{\mathcal Q}=\tilde{L}_1
\end{aligned}$$
where in the last inequality we have used $(\mathcal Q-\mathcal Q_1)/\mathcal Q\le 1$ being $\mathcal Q_2>\mathcal Q_1$. 
Hence, \eqref{moddiscri} is satisfied in this subcase too. 

Finally, to conclude the proof  we are now only left  to show that 
$\gamma$ in \eqref{gamma} is positive. 
Towards this goal, we recall that $t^*>0$ implies
 $b^*>\frac{m-q+1}{\mathcal Q}$.
Now we show that $\upbeta_1>0$. Indeed, from \eqref{beta1}, $\upbeta_1(b)=\frac{b\mathcal Q}{b\mathcal{Q}+q-m}+q-p$
is increasing in $b$ (since $m\leq q$). Therefore, 
\begin{equation}\label{Sep8-1}
\upbeta_1(b^*)\geq \upbeta_1\Bigl(\frac{m-q+1}{\mathcal Q}\Bigr)
=m-q+1+q-p=m+1-p>0,
\end{equation}
where the last inequality follows by the hypothesis of Case 3.

 Next we show that $\upbeta_2>0$ whenever $b^*>1$.
We suppose $b^*>1$. We recall from \eqref{beta2} and \eqref{beta1} that 
\begin{align*}
\beta_2&=\frac{(b-1)(p-q)(m-q)}{b\mathcal{Q}+q-m}+\beta_1\\
&=\frac{(b-1)(p-q)(m-q)}{b\mathcal{Q}+q-m}+\frac{b\mathcal{Q}}{b\mathcal{Q}+q-m}-(p-q).
\end{align*}

Further,
$$\frac{(b-1)(p-q)(m-q)}{b\mathcal{Q}+q-m} + \frac{b\mathcal{Q}}{b\mathcal{Q}+q-m}
= (p-q)\left[\frac{\frac{q-m}{b}+m-q+\frac{\mathcal{Q}}{p-q}}{\frac{q-m}{b}+\mathcal{Q}}\right]:=\xi(b),$$
i.e., $\upbeta_2(b)=\xi(b)-(p-q)$.
It is easy to see that  $t\mapsto \xi(t)$ is a decreasing function for
$m-q+\frac{\mathcal{Q}}{p-q}-\mathcal{Q}\leq 0$ and an increasing
function for $m-q+\frac{\mathcal{Q}}{p-q}-\mathcal{Q}> 0$. Therefore,
if $m-q+\frac{\mathcal{Q}}{p-q}-\mathcal{Q}\leq 0$, we have
$$
\upbeta_2(b^*)\geq \lim_{b\to\infty}\upbeta_2(b)
=1-\frac{(p-q)(1+s)}{m+s-q+1}>0 \quad\mbox{(by hypothesis of the theorem).}
$$
On the other hand, if $m-q+\frac{\mathcal{Q}}{p-q}-\mathcal{Q}> 0$, we have
$$\upbeta_2(b^*)\geq \upbeta_2(1)=\upbeta_1(1)\geq \upbeta_1\Bigl(\frac{m-q+1}{\mathcal Q}\Bigr),$$
where in the last inequality we have used the hypothesis that $s> 0$ implying $\mathcal{Q}>m-q+1$.  Combining the above inequality with \eqref{Sep8-1} we obtain $\upbeta_2(b^*)>0$. Hence, we have proved 
$\upbeta_2(b^*)>0$ if $b^*>1$.

Hence, combining the above with \eqref{Sep8-1} we have shown  
$\gamma>0$ and this completes the proof of \eqref{1-8}. Hence 1st part of the theorem is proved.

(ii) As in the proof of Theorem~\ref{th_HJ}(ii), here also we consider the set $\{z>\varepsilon^\frac{2}{p-q}\}$ while estimating \eqref{Sep11-2}. Doing the same analysis as before (see \eqref{Sep7-1}) will lead  us to 
\begin{equation}\label{Sep11-3}
    \mathcal{H}\geq \kappa_b z ^{1+\min\{1,\upbeta_1\}} \quad\text{on}\quad \{z>\varepsilon^\frac{2}{p-q}\}, 
\end{equation}
with 
$\kappa_b=\min\{\kappa_1, \frac{\kappa_2}{(\varepsilon^{-1}+b^{p-q})^2}\}$ when $b\in(0,1]$. On the other hand when $b>1$, (see \eqref{Sep7-2}) will lead us to
\begin{equation}\label{Sep11-4}
    \mathcal{H}\geq \kappa_b z ^{1+\min\{1,\upbeta_2\}} \quad\text{on}\quad \{z>\varepsilon^\frac{2}{p-q}\}, 
\end{equation}
with 
$\kappa_b=\min\{\kappa_1, \frac{\kappa_2}{(\varepsilon^{-1}+b^{p-q})^2}\}$ .
Hence, setting $\tilde z:=z-\varepsilon^\frac{2}{p-q}$, \eqref{Sep11-1} will be replaced by
$$
\frac12 \mathscr A_v(\tilde z)+\kappa_b \tilde z^{1+\gamma}\le C\frac{|\nabla \tilde z|^2}{\tilde z}\quad\text{on} \quad \{z\geq \varepsilon^\frac{2}{p-q}\}.
$$
Hence by Lemma~\ref{lem1}  it follows
$$z(x)\leq \left(\varepsilon^\frac{2}{p-q}+C\left({\rm dist}\left(x,\partial\Omega\right)\right)^{-\frac{2}{\gamma}}\right),$$
where $C>0$ depends on $\kappa_b$.
Hence, if $\Omega=\mathbb{R}^N$, the above inequality reduces to
$$z(x)\leq \varepsilon^\frac{2}{p-q} \quad\forall\, \varepsilon>0.$$ Taking $\varepsilon\to 0$ we get  $z=0$, i.e., $u$ is constant. 
\hfill{$\square$}

\section{Product nonlinearity with $m>q$ and proof of Theorem~\ref{AB001}}

In this section we prove Theorem~\ref{AB001}. For that,
we first prove that  any  solution (in the sense of Definition~\ref{defi-sol}) is also a viscosity solution at the nondegenerate points.  For details on the definition of viscosity solutions we refer to \cite{CIL}. See also \cite{JLM} for the definition of 
viscosity solution in the context of $p$-Laplacian.
To do this, we introduce the notations
\begin{align*}
F(\xi, M) &= F_p(\xi, M) + F_q(\xi, M),
\\
F_p(\xi, M)&= - |\xi|^{p-2} {\rm tr} M - (p-2) |\xi|^{p-4}\langle\xi M, \xi\rangle,
\\
F_q(\xi, M)&= - |\xi|^{q-2} {\rm tr} M - (q-2) |\xi|^{p-4}\langle\xi M, \xi\rangle,
\end{align*}
where $\xi \in \mathbb R^N$,  $\xi\neq 0$, and $M\in\mathbb{S}_n$, $\mathbb{S}_n$ is the set of all real symmetric $N\times N$ matrices. It is important to note that for any twice differentiable function $u$ we have
$-\Delta_p u(x)=F_p(\nabla u(x), D^2u(x))$ and
$-\Delta_q u(x)=F_q(\nabla u(x), D^2u(x))$, so that $F=\Delta_p+\Delta_q$.

\begin{lemma}\label{L-vis}
Suppose that $w$ is a $C^1$  weak sub-solution in $\Omega$ to $-\mathcal{B}w = f(x, w, \nabla w)$ , where $\mathcal{B}$ is given in Lemma \ref{lemma_Bcal} and $f$ is a continuous function. If for some point $x\in\Omega$, 
there exists a function $\varphi\in C^2(B_r(x))$
, $B_r(x)\Subset \Omega$, such that $\nabla\varphi(x)\neq 0$ and $w(y)-\varphi(y)\leq w(x)-\varphi(x)$ for $y\in B_r(x)$, then we have
$$-\mathcal{B}\varphi\leq f(x, w(x), \nabla \varphi(x)).$$ 
An analogous conclusion holds
for  $C^1$  weak super-solution.
\end{lemma}

\begin{proof}
We prove by contradiction. Suppose that 
\begin{equation}\label{L301}
F(\nabla\varphi(x), D^2\varphi(x))> f(x, w(x), \nabla \varphi(x)).
\end{equation}
Since $\nabla\varphi(x)=\nabla w(x)$,  being $x$ a local minimum for $\varphi-w$, this gives us
$$
F(\nabla\varphi(x), D^2\varphi(x))> f(x, w(x), \nabla w(x)).
$$ 
Therefore, using the continuity of $\varphi$ and $f$, we can find 
$\varepsilon>0$ and $r_1\in (0, r/2]$ such that
$$ 
-\mathcal{B}\varphi(y)=F(\varphi(y), D^2\varphi(y))
\geq f(y, w(y), \nabla w(y)) + \varepsilon, \quad |\nabla\varphi(y)|>0 \quad \text{in}\; \bar{B}_{r_1}(x).
$$
Let $\chi$ be a non-negative smooth function supported in $B_{r_1}(x)$ and $\chi(x)= 1$, where $x$ is given as above. Define $\varphi_\theta(y)=\varphi(y)-\theta\chi(y) + w(x)-\varphi(x)$ for $\theta\in (0, 1)$. Using continuity we can
find $\tilde\theta\in (0, 1)$ such that 
$$
-\mathcal{B}\varphi_{\tilde\theta}(y)\geq f(y, w(y), \nabla w(y)) +\varepsilon/2 \quad \text{in}\; \bar{B}_{r_1}(x).$$
Thus, we obtain
\begin{equation}\label{ineq_eps/2}-\mathcal{B} w + \mathcal{B}\varphi_{\tilde\theta}(y)\leq -\varepsilon/2\quad \text{in}\; \bar{B}_{r_1}(x),\end{equation}
in the weak sense. Consider the test function $v=(w-\varphi_{\tilde\theta})_+$. Since $w\leq \varphi+ w(x)-\varphi(x)= \varphi_{\tilde\theta}$ on $\partial B_{r_1}(x)$, we have 
$v\in W^{1,p}_0(B_{r_1}(x))$. Now multiply \eqref{ineq_eps/2} by $v$ and perform an integration by parts to arrive at
\begin{align*}
-\frac{\varepsilon}{2}\int_{B_{r_1}(x)} v(y) \dy
&\geq \int_{B_{r_1}(x)} (|\nabla w|^{p-2}\nabla w(y)-|\nabla \varphi_{\tilde\theta}|^{p-2}\nabla\varphi_{\tilde\theta}(y))\cdot \nabla v(y) \dy 
\\
&\qquad + \int_{B_{r_1}(x)} (|\nabla w|^{q-2}\nabla w(y)-|\nabla \varphi_{\tilde\theta}|^{q-2}\nabla\varphi_{\tilde\theta}(y))\cdot \nabla v(y) \dy 
\\
&=  \int_{B_{r_1}(x)\cap \{v>0\}} (|\nabla w|^{p-2}\nabla w(y)-|\nabla \varphi_{\tilde\theta}|^{p-2}\nabla\varphi_{\tilde\theta}(y))\cdot (\nabla w(y)-\nabla \varphi_{\tilde\theta}(y)) \dy 
\\
&\qquad+ \int_{B_{r_1}(x)\cap \{v>0\}} (|\nabla w|^{q-2}\nabla w(y)-|\nabla \varphi_{\tilde\theta}|^{q-2}\nabla\varphi_{\tilde\theta}(y))\cdot (\nabla w(y)-\nabla \varphi_{\tilde\theta}(y)) \dy
\\
&\geq 0,
\end{align*}
where we  used monotonicity 
of the maps $t\mapsto |t|^{p-2}t, |t|^{q-2}t$. Thus we get
$\int_{B_{r_1}(x)} v(y)=0$, implying $v\equiv 0\Rightarrow w\leq \varphi_{\tilde\theta}$ in $B_{r_1}(x)$. But $w(x)>\varphi_{\tilde\theta}(x)$, which is a contradiction. Hence \eqref{L301} can not hold, completing 
the proof.
\end{proof}
At this point we introduce the notion of subjet $J^{2,+}$ and superjet $J^{2, -}$ which are defined as follows
\begin{align*}
J^{2,+} u(x)=\{(\nabla \varphi(x),D^2\varphi(x)) &:\; \varphi\; \text{is $C^2$ in a neighbourhood of $x$ }
\\
&\quad \text{and $u-\varphi$ has a local maximum at $x$}\},
\end{align*}
and 
\begin{align*}
J^{2,-} u(x)=\{(\nabla \varphi(x),D^2\varphi(x)) &:\; \varphi\; \text{is $C^2$ in a neighbourhood of $x$ }
\\
&\quad \text{and $u-\varphi$ has a local minimum at $x$}\}.
\end{align*}
The closure of these jets are defined as follows: for $x\in\Omega$
\begin{align*}
\bar{J}^{2, +} u(x)=\{(\xi, X)\in \mathbb{R}^n\times\mathbb{S}_n\; &:\; \exists\, (x_n, \xi_n, X_n)\in \Omega\times\mathbb{R}^n\times\mathbb{S}_n 
\; \text{with}\; (\xi_n, X_n)\in J^{2,+}u(x_n)
\\
&\quad \text{and}\; (x_n, u(x_n), \xi_n, X_n)\to (x, u(x), \xi, X)\},
\end{align*}
and
\begin{align*}
\bar{J}^{2, -} u(x)=\{(\xi, X)\in \mathbb{R}^n\times\mathbb{S}_n\; &:\; \exists\, (x_n, \xi_n, X_n)\in \Omega\times\mathbb{R}^n\times\mathbb{S}_n 
\; \text{with}\; (\xi_n, X_n)\in J^{2,-}u(x_n)
\\
&\quad \text{and}\; (x_n, u(x_n), \xi_n, X_n)\to (x, u(x), \xi, X)\}.
\end{align*}
From Lemma~\ref{L-vis}, we know that if $u$ is a solution to $-\mathcal{B} u=f(u, \nabla u)$, then $u$ is a viscosity solution to 
$-\mathcal{B} u=f(u, \nabla u)$ at the points $\nabla u(x)\neq 0$. Again, since $\nabla u(x)=\nabla \varphi(x)$ for any
$(\nabla \varphi(x),D^2\varphi(x))\in J^{2,+}u(x)$, we have $F(\varphi(x), D^2\varphi(x))\leq  f(u(x), \nabla \varphi(x))$ for 
$(\nabla \varphi(x),D^2\varphi(x))\in J^{2,+}u(x)$ and $\nabla u(x)\neq 0$. Using the continuity of $F$, it is easily seen that
\begin{equation}\label{test1}
F(\xi, X)\leq  f(u(x), \xi)\quad \text{for}\;\; (\xi,X)\in \bar{J}^{2,+}u(x)\quad \text{and}\quad \nabla u(x)\neq 0.
\end{equation}
Considering now $u$ as a weak supersolution, an analogous conclusion also holds for
$\bar{J}^{2,-}$, that is,
\begin{equation}\label{test2}
F(\xi, X)\geq  f(u(x), \xi)\quad \text{for}\;\; (\xi,X)\in \bar{J}^{2,-}u(x)\quad \text{and}\quad \nabla u(x)\neq 0.
\end{equation}
This observation will be used below while applying Crandall-Ishii-Jensen lemma \cite[Theorem~3.2]{CIL}.

\begin{proof}[{\bf Proof of Theorem~\ref{AB001}}]
Consider a smooth cut-off function $\psi:\mathbb{R}^n\to [0, \infty)$ satisfying $\psi(x)=0$ for $|x|\leq \frac{1}{4}$, 
$\psi(x)=2\norm{u}_\infty$ for $|x|\geq \frac{1}{2}$ and $0\leq \psi\leq 2\norm{u}_\infty$. We fix $\gamma\in (0, 1)$ so that
\begin{equation}\label{gamma_range} m-1 - (1-\gamma)(q-1)> (q-1)\gamma\quad  \text{and}\quad 0< \frac{(\gamma-1)(m+1-q)+1}{m+1-q}<\gamma.\end{equation}
 Also, consider a function $\delta:(0, \infty)\to (0, \infty)$
satisfying the following
$$\lim_{r\to\infty}\delta(r)=0, \quad \lim_{r\to\infty} \delta(r) \left(\frac r4\right )^\gamma=\infty,\quad
\text{and}\quad \lim_{r\to\infty} \delta(r) r^{\frac{(\gamma-1)(m+1-q)+1}{m+1-q}}=0.$$
For $R>1$, let us define the doubling function 
$$\Phi(x, y)=u(x)-u(y)-\delta |x-y|^\gamma - \psi\left(\frac xR\right)$$
with $\delta=\delta(R)$. We claim that there exists a $R_0$ satisfying 
\begin{equation}\label{AB002}
\sup_{(x, y)\in B_R\times B_R} \Phi(x, y)\leq 0 \quad \text{for all}\; R\geq R_0.
\end{equation}
Once \eqref{AB002} is established, we can complete the proof as follows: fix any $x, y\in\mathbb{R}^N$ and consider any $R\geq R_0$ satisfying $|x|, |y|\leq R/4$. From \eqref{AB002} we then have $|u(x)-u(y)|\leq \delta(R)|x-y|^\gamma$. Now letting $R\to\infty$ and using the first property of  $\delta$, we see that $u(x)=u(y)$, so $u$ turns out to be a constant.

We prove \eqref{AB002} by contradiction. We start by assuming that for some large $R$
$$\Phi(\bar{x},\bar{y})=\max_{(x, y)\in \bar{B}_R\times \bar{B}_R} \Phi(x, y)>0,$$
where $\bar{x},\bar{y}\in \bar{B}_R$. By the definition of $\psi$ for $|x|\ge R/2$, we have $\bar{x}\in B_{R/2}$. From the second property of $\delta$,  it also follows that
$|\bar{x}-\bar{y}|\leq \frac{R}{4}$ for all large $R$. We set $\bar{a}=\bar{x}-\bar{y}$. Since $\Phi(\bar{x},\bar{y})>0$, it follows that $\bar{a}\neq 0$.
We denote by
\begin{align}\label{phi-pq}
\phi(x, y)&:=\delta |x-y|^\gamma+\psi(x/R), \quad \bar{p}:=\nabla_x\phi(\bar{x},\bar{y}) = \delta\gamma |\bar{a}|^{\gamma-2} \bar{a}+ R^{-1} \nabla\psi(\bar{x}/R), \nonumber
\\
\bar{q}&:=-\nabla_y\phi(\bar{x},\bar{y})=\delta\gamma |\bar{a}|^{\gamma-2}\bar{a}.
\end{align}
Note that, since $|\bar{a}|\leq R/4$ and $\gamma<1$, we have $|\bar{p}|\geq \gamma \delta(R) R^{\gamma-1} - R^{-1}|\nabla\psi|_\infty = R^{-1} (\gamma \delta(R) R^{\gamma} - |\nabla\psi|_\infty)>0$
for all large $R$.
Applying \cite[Theorem 3.2]{CIL}, we see that for any $\varepsilon>0$, there exists $X, Y\in \mathbb{S}_n$ satisfying
\begin{align}\label{pX}
(\bar{p}, X)\in \bar{J}^{2, +} u(\bar{x}), \quad (\bar{q}, Y)\in \bar{J}^{2, -} u(\bar{x}),
\end{align}
and 
\begin{equation}\label{AB003}
-({\varepsilon}^{-1} + \norm{D^2 \phi(\bar{x}, \bar{y})}) I\leq
\begin{pmatrix}
X & 0\\
0 & -Y
\end{pmatrix}
\leq D^2 \phi(\bar{x}, \bar{y}) + \varepsilon (D^2 \phi(\bar{x}, \bar{y}))^2.
\end{equation}
Here $\norm{D^2 \phi(\bar{x}, \bar{y})}$ denotes the maximum of the modulus of the eigenvalues of $D^2 \phi(\bar{x}, \bar{y})$. Letting  $M=D^2 \varphi(\bar{a})$ with $\varphi(x)=\delta |x|^\gamma$, we note that
$$ 
D^2 \phi(\bar{x}, \bar{y})=
\begin{pmatrix}
M & -M\\
-M & M
\end{pmatrix}
+ R^{-2}
\begin{pmatrix}
D^2\psi(\bar{x}/R) & 0\\
0 & 0
\end{pmatrix}.
$$
For our calculations below, we set $\varepsilon=\kappa \delta^{-1} |\bar{a}|^{2-\gamma}$ for some $\kappa\in (0,1)$ to be chosen later. With these notations in hand, we see from \eqref{AB003},   multiplying by $(v, 0)$ 
for any unit vector $v\in\mathbb{R}^n$ we have, because of the structure of $M_1=D^2 \phi(\bar{x}, \bar{y})$, that
$$ - (\kappa^{-1}\delta|\bar{a}|^{\gamma-2} +\norm{M_1}) \leq \langle v X, v\rangle \leq  \langle v M, v\rangle + C [R^{-2} + \kappa\delta^{-1} |\bar{a}|^{2-\gamma} (\norm{M}^2 + R^{-2} \norm{M}+R^{-4})],$$
where the constant $C$ depends only on $\norm{D^2 \psi}_\infty$. Note that, by the property of $\delta$,
$$
R^2= R^{2-\gamma} R^{\gamma}\geq  \delta^{-1}|\bar{a}|^{2-\gamma},
$$
provided $R$ is large, giving us $R^{-2}\leq \delta |\bar{a}|^{\gamma-2}$. This also
implies
$$\norm{M_1}=\norm{D^2 \phi(\bar{x}, \bar{y})}\leq \kappa_1 \delta |\bar{a}|^{\gamma-2},$$
for some constant $\kappa_1$, independent of $R, \bar{x}, \bar{y}$.  Since all the norms are equivalent in finite dimension, the above estimate is obtained by estimating the entries of $M_1$.
Thus, we obtain
\begin{equation}\label{AB004}
\norm{X} \leq \kappa_2 \delta |\bar{a}|^{\gamma-2},
\end{equation}
for some constant $\kappa_2$. Using \eqref{test1}-\eqref{test2} and \eqref{pX}, we see that
$$F(\bar{p}, X)\leq f(u(\bar{x}), \bar{p})\quad \text{and} \quad F(\bar{q}, Y)\geq f(u(\bar{y}), \bar{q}).$$
Subtracting the above inequalities we arrive at
\begin{equation}\label{AB005}
\underbrace{F_p(\bar{p}, X)-F_p(\bar{q}, Y)}_{I}+\underbrace{F_q(\bar{p}, X)-F_q(\bar{q}, Y)}_{II}=F(\bar{p}, X)-F(\bar{q}, Y) \leq f(u(\bar{x}), \bar{p}) - f(u(\bar{y}), \bar{q}).
\end{equation}
Next, we compute $I$ and $II$. Because of similarity, we only provide the details for $I$.
We write
\begin{align}\label{AB006}
-I &= \underbrace{(|\bar{p}|^{p-2} -|\bar{q}|^{p-2}){\rm tr} X + (p-2) [|\bar{p}|^{p-4}\langle \bar{p} X,\bar{p}\rangle - |\bar{q}|^{p-4}\langle \bar{q} X,\bar{q}\rangle]}_{I_1}\nonumber
\\
&\qquad + \underbrace{|\bar{q}|^{p-2}{\rm tr} (X-Y) + (p-2) |\bar{q}|^{p-4}\langle \bar{q} (X-Y),\bar{q}\rangle}_{I_2}.
\end{align}
First consider $I_2$. Since $\bar{q}\neq 0$, we consider a orthonormal basis  in $\mathbb R^N$ given by $(\frac{\bar{q}}{|\bar{q}|}, v_1, \ldots, v_{n-1})$ and notice that
\begin{align*}
I_2 &=  ( |\bar{q}|^{p-4}\langle \bar{q} (X-Y),\bar{q}\rangle + \sum_{i=1}^{n-1}|\bar{q}|^{p-2}\langle v_i(X-Y), v_i\rangle) + (p-2) |\bar{q}|^{p-4}\langle \bar{q} (X-Y),\bar{q}\rangle
\\
&= (p-1) |\bar{q}|^{p-4}\langle \bar{q} (X-Y),\bar{q}\rangle + \sum_{i=1}^{n-1}|\bar{q}|^{p-2}\langle v_i(X-Y), v_i\rangle).
\end{align*}
Applying \eqref{AB003} on the vector $(v_i, v_i)$, we see that 
$$\langle v_i(X-Y), v_i\rangle\leq C(R^{-2} + \kappa \delta |\bar{a}|^{\gamma-2})$$
for some constant $C$, dependent only on $\norm{D^2\psi}_\infty$. We also apply \eqref{AB003} on the vector $(\bar{q}, -\bar{q})$ to obtain
$$\langle \bar{q} (X-Y),\bar{q}\rangle\leq 4 \langle \bar{q} M,\bar{q}\rangle + C (R^{-2} + \kappa\delta |\bar{a}|^{\gamma-2})|\bar{q}|^2.$$
Since 
$$M_{ij}=\delta\gamma(\gamma-1)|\bar{a}|^{\gamma-2} \frac{(\bar{a})_i(\bar{a})_j}{|\bar{a}|^2} + \delta \gamma |\bar{a}|^{\gamma-2}\left(\delta_{ij}-\frac{(\bar{a})_i(\bar{a})_j}{|\bar{a}|^2}\right),$$
from the definition of $\bar{q}$ we have
$$\langle \bar{q} M,\bar{q}\rangle
= \delta \gamma (\gamma-1) |\bar{q}|^2 |\bar{a}|^{\gamma-2}. $$
Thus
$$I_2\leq \left[4(p-1)\delta\gamma(\gamma-1) |\bar{a}|^{\gamma-2} + C(R^{-2} + \kappa\delta |\bar{a}|^{\gamma-2})\right]|\bar{q}|^{p-2}.$$
Now we choose $\kappa\in (0, 1)$, dependent on $R$, small enough so that $\kappa C=2  (p-1)\gamma(1-\gamma)$. Thus, we have
$$I_2\leq \delta \left[-2(p-1)\gamma(1-\gamma) |\bar{a}|^{\gamma-2} + CR^{-2}\delta^{-1} \right]|\bar{q}|^{p-2}.$$
Since $R^{-2}\delta^{-1}|\bar{a}|^{2-\gamma}\leq R^{-\gamma} \delta^{-1}\to 0$, as $R\to\infty$ and using the definition of $\bar q$ from \eqref{phi-pq}, we get
$$
I_2\leq -(p-1)\gamma(1-\gamma)\delta |\bar{q}|^{p-2} |\bar{a}|^{\gamma-2}= -(p-1)(1-\gamma) \gamma^{p-1} \delta^{p-1} |\bar{a}|^{(p-1)(\gamma-1)-1}
$$
for all large $R$.

To compute $I_1$,  we first observe that as $|\bar a|<\frac{R}{4}$ implies $R\geq |\bar{a}|^{1-\gamma}R^\gamma> \frac{1}{C} \delta^{-1} |\bar{a}|^{1-\gamma}$ for large enough $R$. Therefore, using the definition of $\bar p$ and $\bar q$ from \eqref{phi-pq}, we obtain 
\begin{equation}\label{p_bar_q_bar} |\bar{p}-\bar{q}|\leq C R^{-1}, \quad  |\bar{p}|\leq C \delta |\bar{a}|^{\gamma-1}. \quad\end{equation}
 Again, for any exponent $\beta\in\mathbb{R}$, {by the definition of $\bar p$ and $\bar q$,}
\begin{align*}
\left||\bar{p}|^{\beta}-|\bar{q}|^\beta\right|=(\gamma\delta|\bar{a}|^{\gamma-1})^{\beta} \left(|\hat{a}+ \frac{|\bar{a}|^{1-\gamma}}{\gamma R\delta}\nabla\psi(\bar{x}/R)|^\beta-1 \right),
\end{align*}
where $\hat{a}$ denotes the unit vector along $\bar{a}$. Since  $|\bar a|\le R$, being $|\bar x|, |\bar y|\le R/2,$ then, by the properties of $\delta$,
$$
\frac{|\bar{a}|^{1-\gamma}}{ R\delta}\leq \frac{1}{R^\gamma\delta(R)}\xrightarrow{R\to\infty} 0,
$$
so that for $R$ large the vector $\hat{a}+ \frac{|\bar{a}|^{1-\gamma}}{\gamma R\delta}\nabla\psi(\bar{x}/R)$ is close in norm  to the unit vector $\hat a=\bar a/|\bar a|$. In turn, 
using the Lipschitz property of $x\mapsto |x|^\beta$ around $|x|=1$,  namely,  for any unit vector $e$, the map $x\mapsto |x+e|^\beta$ is Lipschitz for $|x|<1/2$, as $t\mapsto t^\beta$ is smooth for $t>1/2$,  we have
\begin{equation}\label{beta_estim} \left||\bar{p}|^{\beta}-|\bar{q}|^\beta\right|\leq (\gamma\delta|\bar{a}|^{\gamma-1})^{\beta} \zeta(R)\end{equation}
for some function $\zeta$ that vanishes as $R\to\infty$. Now using  \eqref{beta_estim}  with $\beta=p-2$ and \eqref{AB004}, we see that for all $R$ large enough we have
 $$|(|\bar{p}|^{p-2} -|\bar{q}|^{p-2}){\rm tr} X|\le  \kappa_2\delta (\gamma\delta)^{p-2} |\bar{a}|^{(p-1)(\gamma-1)-1}\zeta(R).$$
Next, using \eqref{beta_estim} with $\beta=p-4$, \eqref{AB004}  and \eqref{p_bar_q_bar}  we estimate
 $$
 \begin{aligned}
 \bigl||\bar{p}|^{p-4}\langle \bar{p} X,\bar{p}\rangle - 
 |\bar{q}|^{p-4}\langle \bar{q} X,\bar{q}\rangle\bigr|
 &\le |\langle \bar{p} X,\bar{p}\rangle \bigl| |\bar{p}|^{p-4} - |\bar{q}|^{p-4}\bigr| + |\bar{q}|^{p-4}
 |\langle \bar{p} X,\bar{p}\rangle-\langle \bar{q} X,\bar{q}\rangle|
 \\
 & \le |\bar p|^2\kappa _2\delta |\bar a|^{\gamma-2}(\gamma\delta|\bar{a}|^{\gamma-1})^{p-4} \zeta(R)+ 
 |\bar{q}|^{p-4}(|\bar{p}|+|\bar{q}|)\norm{X}\frac{1}{R}
 \\
 &=C \delta (\gamma\delta)^{p-2} |\bar{a}|^{(p-1)(\gamma-1)-1}\zeta(R) +  C \delta^{p-1}|\bar{a}|^{(p-1)(\gamma-1)-1}\tilde\zeta(R),
\end{aligned},$$
where in the second term of the last equality we have used
$\frac{1}{R}\leq \delta |\bar{a}|^{\gamma-1}\frac{1}{\delta R^\gamma}=\delta |\bar{a}|^{\gamma-1} \tilde\zeta(R)$ with $\tilde\zeta(R)\to 0$ as $R\to\infty$. 
Therefore, for $R$ sufficiently large, we have
 $$I_1\leq \frac{1}{2} (p-1)(1-\gamma) \gamma^{p-1} \delta^{p-1} |\bar{a}|^{(p-1)(\gamma-1)-1}.$$
Plugin the  estimates of $I_1$ and $I_2$ in \eqref{AB006} we obtain
$$I\geq \frac{1}{2} (p-1)(1-\gamma) \gamma^{p-1} \delta^{p-1} |\bar{a}|^{(p-1)(\gamma-1)-1},$$
provided $R\geq R_0$, where $R_0$ is chosen large depending on the estimates above. Similarly, we would also have 
$$ II\geq \frac{1}{2} (q-1)(1-\gamma) \gamma^{q-1} \delta^{q-1} |\bar{a}|^{(q-1)(\gamma-1)-1}, $$
provided $R\geq R_0$. 

Now,  by \eqref{f_IL}, since $u$ is bounded  and using that $\bar p=\bar q+\nabla\psi/R$, we have
$$f(u(\bar{x}), \bar{p}) - f(u(\bar{y}), \bar{q})\leq C(|\bar{q}|+R^{-1})^m\leq C_1 [\delta^m |\bar{a}|^{m(\gamma-1)} + R^{-m}]. $$
Since $I\geq 0$, from \eqref{AB005} we obtain
\begin{equation}\label{AB007}
\frac{1}{2} (q-1)(1-\gamma) \gamma^{q-1} \delta^{q-1} |\bar{a}|^{(q-1)(\gamma-1)-1}\leq C_1 [\delta^m |\bar{a}|^{m(\gamma-1)} + R^{-m}],
\end{equation}
provided $R\geq R_0$. We observe that
\begin{align*}
\delta^m |\bar{a}|^{m(\gamma-1)} &= \delta^{q-1} \delta^{m+1-q}R^{m(\gamma-1)}\left(\frac{|\bar{a}|}{R}\right)^{m(\gamma-1)}
\\
&\leq \delta^{q-1} \delta^{m+1-q}R^{m(\gamma-1)}\left(\frac{|\bar{a}|}{R}\right)^{(q-1)(\gamma-1)-1}
\\
&= \delta^{q-1} [\delta R^{\frac{(m+1-q)(\gamma-1)+1}{m+1-q}}]^{m+1-q} |\bar{a}|^{(q-1)(\gamma-1)-1},
\end{align*}
where in the second line we used the fact $m(\gamma-1)> (q-1)(\gamma-1)-1$, by \eqref{gamma_range}. Now, by the second property of $\delta$, we can choose $R_0$ large enough so that
$$\frac{1}{4} (q-1)(1-\gamma) \gamma^{q-1} \delta^{q-1} |\bar{a}|^{(q-1)(\gamma-1)-1}\geq
\frac{1}{4} (q-1)(1-\gamma) \gamma^{q-1} \delta^{q-1} R^{(q-1)(\gamma-1)-1}\geq C_1 R^{-m},$$
for all $R\geq R_0$, where we use the fact
$$ \delta^{q-1} R^{(q-1)(\gamma-1)+m-1}=
[\delta(R) R^{\gamma +\frac{m-1}{q-1}-1}]^{q-1}\to\infty,
\quad \text{as}\; R\to\infty,$$
using $m>q$.

Hence, inserting the above estimate in \eqref{AB007}, we get for $R\geq R_0$  
$$
 \frac{1}{4} (q-1)(1-\gamma) \gamma^{q-1} \delta^{q-1} |\bar{a}|^{(q-1)(\gamma-1)-1}\leq C_1 \delta^{q-1} [\delta R^{\frac{(m+1-q)(\gamma-1)+1}{m+1-q}}]^{m+1-q} |\bar{a}|^{(q-1)(\gamma-1)-1}
$$
so that for  some
constant $C_2$ it must hold
$$0< C_2\leq \delta R^{\frac{(m+1-q)(\gamma-1)+1}{m+1-q}}.$$
By the third property of $\delta$, this cannot occur for all large $R$.
Hence we have reached a contradiction. This gives us the claim \eqref{AB002} and completes the proof.
\end{proof}

\section{Sum of nonlinearities: Proof of Theorems~\ref{thm+} and \ref{thm+_v^b}}
We first observe that the so called critical exponent with respect to the gradient for equation \eqref{main_M+}, given by $ps/(s+1)$ for the $p$-Laplacian case as discussed  in \cite{FSZ}, see also \cite{VB2} for $p=2$, continues to be the same. This suggests that, in the context of the  $(p,q)$-Laplacian with $q<p$, the  $p$-Laplacian operator is, in some sense, the dominant one.

To this aim, note that if $v(x)=k^\alpha u(kx)$, $k>0$, then
by routine calculation $\Delta_p v(x)=k^{\alpha(p-1)+p}\Delta_p u(y)$ with $y=kx$, so that
if $u$ is a solution of \eqref{main_M+}, it follows that $v$ is a solution of 
$$-\Delta_pv-k^{(\alpha+1)(p-q)}\Delta_qv=k^{\alpha(p-1-s)+p}v^s+Mk^{\alpha(p-1-m)+p-m}|\nabla v|^m.$$
In particular, if $u$ is a solution of \eqref{main_M+}, then   $v(x)=k^{p/(s-p+1)}u(kx)$ is a solution of 
$$-\Delta_pv-k^{\frac{(s+1)(p-q)}{s-p+1}}\Delta_qv=v^s+Mk^{(s+1)\bigl(\frac{sp}{s+1}-m\bigr) }\,|Dv|^m,$$
so that, in the subcritical case,  by letting $k\to0$ we recover $-\Delta_pv=v^s$.

On the other hand, if $v(x)=k^{(p-m)/(m-p+1)}u(kx)$, with $u$  solution of \eqref{main_M+}, then $v$ is a solution of 
$$-\Delta_pv-k^{\frac{p-q}{m-p+1}}\Delta_qv=k^{\frac{1+s}{m-p+1}\bigl(m-\frac{sp}{s+1}\bigr)} \, v^s+M|Dv|^m,$$
so that, in the supercritical case,  by letting $k\to0$ we recover $-\Delta_pv=M|\nabla v|^m$.

\bigskip{\it {\bf Proof of Theorem \ref{thm+}}.} We start the proof by taking inequality \eqref{ineq_b=1} with $\mathcal Bu=-u^s-M|\nabla u|^m$, being
$$\begin{aligned}
    \frac{(\mathcal Bu)^2}{z^{q-2}}&=\frac{u^{2s}}{z^{q-2}}
+M^2z^{m-q+2}+2Mu^sz^{\frac m2-q+2}
\\\frac{\mathcal Bu \langle\nabla z,\nabla u\rangle}{z^{q/2}}&=-\bigl(u^sz^{-q/2} +M z^{(m-q)/2}\bigr)\langle\nabla z,\nabla u\rangle\\\frac{\langle\nabla\bigl(\mathcal Bu\bigr),\nabla v\rangle}{z^{(q-2)/2}}&= -su^{s-1}z^{(4-q)/2}-M\frac m2 z^{(m-q)/2}\langle\nabla z,\nabla u\rangle,
\end{aligned}$$
we reach, by Lemma \ref{lemma_Bcal} and estimate \eqref{ineq_C_1}, 
\begin{equation}\label{est1_th+}\begin{aligned} \frac12\mathscr{A}_v(z)& +\frac 1{N A^2}\biggl[u^{2s}z^{2-q}
+M^2z^{m-q+2}+2Mu^sz^{\frac m2-q+2}\biggr]-\frac1Asu^{s-1}z^{(4-q)/2}\\&\le 
-\frac1{A}\biggl(\frac1N+\frac12\biggr)\frac D{A}\bigl(u^sz^{-q/2} +M z^{(m-q)/2}\bigr)\langle\nabla z,\nabla v\rangle
\\&\quad +\frac MA\frac m2 z^{(m-q)/2}\langle\nabla z,\nabla u\rangle
+ C_2\frac{|\nabla z|^2}{z}.
\end{aligned}\end{equation}
Next, by \eqref{estim_DA_EA}, applying Young inequality, we have
$$\biggl|\frac1{A}\biggl(\frac1N+\frac12\biggr)\frac D{A}\biggr|u^sz^{-q/2} \bigl|\langle\nabla z,\nabla v\rangle\bigr|\le \frac\varepsilon {A^2} u^{2s} z^{2-q}+C_\varepsilon \frac{|\nabla z|^2}z$$
and 
$$\frac M{A}\biggl|\frac m2-\frac1{A}\biggl(\frac1N+\frac12\biggr)\frac D{A}\biggr| z^{(m-q)/2}\bigl|\langle\nabla z,\nabla v\rangle\bigr|\le \frac\varepsilon {A^2} z^{m-q+2}+C_\varepsilon \frac{|\nabla z|^2}z.$$
Consequently, inequality \eqref{est1_th+} becomes
\begin{equation}\label{est2_th+}
\begin{aligned}\frac12\mathscr{A}_v(z)+\biggl(\frac 1N-\varepsilon&\biggr)\frac{u^{2s}z^{2-q}}{A^2}+\biggl(\frac {M^2}N-\varepsilon\biggr)\frac{z^{m-q+2}}{A^2}\\&
+2\frac M N \frac{u^sz^{\frac m2-q+2}}{A^2}-\frac1Asu^{s-1}z^{(4-q)/2}\le C\frac{|\nabla z|^2}{z}\quad\mbox{on }\{z>0\}
\end{aligned}\end{equation}
for some positive constant $C$.
Note that hypothesis $s>q-1$ and $m>\frac{sq}{s+1}$ implies $m-q+2>1$.
Further, as $s>1$, applying Young inequality with exponents 
$2s/(s-1)$ and $2s/(s+1)$ we have
$$\frac1Asu^{s-1}z^{(4-q)/2}=\frac{su^{s-1}z^{(2-q)(s-1)/2s}}{A^{(s-1)/s}}\cdot \frac{z^{1+(2-q)/2s}}{A^{1/s}}
\le \varepsilon \frac{u^{2s}z^{2-q}}{A^2}+C_\varepsilon \frac{z^{\frac{2s+2-q}{s+1}}}{A^{2/(s+1)}}.$$
Thanks to \eqref{bound+}, a further application of Young inequality with exponents 
$$\frac{(m-q+2)(s+1)}{2s+2-q}\quad\mbox{and}\quad
\frac{(m-q+2)(s+1)}{s(m-q)+m}$$
gives
$$C_\varepsilon \frac{z^{\frac{2s+2-q}{s+1}}}{A^{\frac 2{s+1}}}=
C_\varepsilon \frac{z^{\frac{2s+2-q}{s+1}}}{A^{\frac 2{s+1}\cdot \frac{2s+2-q}{m-q+2}}\cdot A^{\frac 2{s+1}\cdot \frac{m-2s}{m-q+2}} }\le 
\varepsilon\frac{z^{m-q+2}}{A^2}
+C_\varepsilon \frac 1{A^{\frac{2(m-2s)}{s(m-q)+m}}}.$$
Inserting the above estimates in \eqref{est2_th+}, 
it follows
$$\begin{aligned}\frac12\mathscr{A}_v(z)+\biggl(\frac 1N-2\varepsilon\biggr)\frac{u^{2s}z^{2-q}}{A^2}&+\biggl(\frac {M^2}N-2\varepsilon\biggr)\frac{z^{m-q+2}}{A^2}
\\&\le \frac{|\nabla z|^2}{z}+C_\varepsilon \frac 1{A^{\frac{2(m-2s)}{m(s+1)-qs}}}\quad\mbox{on }\{z>0\},\end{aligned}$$
yielding for $\varepsilon$ sufficiently small
$$\mathscr{A}_v(z)+c_1\frac{z^{m-q+2}}{A^2}
\le C\frac{|\nabla z|^2}{z}+C_1 \frac 1{A^{\frac{2(m-2s)}{m(s+1)-qs}}}\quad\mbox{on }\{z>0\},$$
Now, by \eqref{est2_th+}, the exponent of $A$ is positive, so that we obtain
$$\mathscr{A}_v(z)+c_1\frac{z^{m-q+2}}{\bigl(z^{\frac{p-q}2}+1\bigr)^2}
\le C\frac{|\nabla z|^2}{z}+C_1 \quad\mbox{on }\{z>0\}.$$
Now, note that on the set $\{z>1\}$ the inequality  $2z^{\frac{p-q}2}>1+z^{\frac{p-q}2}$ holds , which in particular gives $\bigl(z^{\frac{p-q}2}+1\bigr)^{-2}>z^{p-q}/4$, so that the following  is in force
$$\mathscr{A}_v(z)+\frac{c_1}4z^{m-p+2}
\le C\frac{|\nabla z|^2}{z}+C_1 \quad\mbox{on }\{z>1\}.$$
Set $\tilde z:=z-1$, then 
$$\frac12\mathscr{A}_v(\tilde z)+\frac{c_1}4{\tilde z}^{m-p+2}
\le C\frac{|\nabla \tilde z|^2}{\tilde z}+C_1\quad\mbox{on }\{\tilde z>0\}.$$


Therefore, applying \cite[Lemma 3.1]{BV} we obtain 
$${\tilde z}\le C\big(1+ \left({\rm dist}\left(x,\partial\Omega\right)\right)^{-\frac{2}{m-p+2}}\big),$$
in turn \eqref{1-2} follows at once.
\hfill$\square$

\bigskip{\it{\bf  Proof of Theorem \ref{thm+_v^b}}.}
Let $u$ be a solution of \eqref{main_M+}. As in the proof of Theorem \ref{th_sm}, here we need to consider the change of variables $u=v^b$ so that  the inequality in the statement of  Lemma \ref{lemma_Bcal}, when
$$\mathcal Bu=-u^s-M|\nabla u|^m=-v^{bs}-M|b|^mv^{m(b-1)}z^{m/2},
$$ and
$$\begin{aligned}\frac{\bigl(\mathcal Bu\bigr)^2}{b^{2(q-1)}}
\frac1{v^{2(b-1)(q-1)}z^{q-2}}=&\frac1{b^{2(q-1)}}v^{2[b(s-q+1)+q-1]}z^{2-q}+M^2|b|^{2(m-q+1)}v^{2(b-1)(m-q+1)}z^{m-q+2}\\&+2M|b|^{m-2(q-1)}v^{(b-1)[m-2(q-1)]+bs}z^{\frac m2-q+2}
\end{aligned}$$
$$\begin{aligned}-\frac{\langle\nabla\bigl(\mathcal Bu\bigr),\nabla v\rangle}{b|b|^{q-2}}
\frac 1{v^{(b-1)(q-1)}z^{\frac{q-2}2}}=&\frac s{|b|^{q-2}}v^{b(s-q+1)+q-2}z^{2-\frac q2}+\frac {Mm(b-1)}{b|b|^{q-m-2}}v^{(b-1)(m-q+1)-1}z^{ 2+\frac{m-q}{2}}
\\&+\frac {Mm}{2b|b|^{q-2-m}}v^{(b-1)(m-q+1)}z^{\frac{m-q}2}
\langle\nabla z,\nabla v\rangle,
\end{aligned}$$

so that, denoting with
\begin{equation}\label{tau}\tau=b(s-q+1)+q-1,\end{equation}
by Lemma \ref{lemma_Bcal}, together with  \eqref{ineq_C_1},  \eqref{C_1} and \eqref{nabla/v}, we have
$$\begin{aligned} \frac12\mathscr{A}_v&(z) + (A_1-\varepsilon)\frac{z^2}{v^2}+\frac1{Nb^{2(q-1)}}\frac1{A^2}{v^{2\tau}z^{2-q}}+\frac {M^2}N|b|^{2(m-q+1)}\frac1{A^2}{v^{2(b-1)(m-q+1)}z^{m+2-q}}\\&\quad +2\frac MN|b|^{m-2(q-1)}\frac1{A^2}{v^{(b-1)[m-2(q-1)]+bs}z^{\frac m2+2-q}}\\&\le 
-\biggl(\frac1N+\frac 12\biggr)\frac D{A}\frac{v^{\tau}z^{-\frac q2}}{b|b|^{q-2}}
\frac{\langle\nabla z,\nabla v\rangle}A
+\frac 1{A|b|^{q-2}}\biggl\{s-
2\frac{b-1}{b} \biggl(\frac 1N+\frac 12\biggr)\frac EA\biggr\}v^{\tau-1}z^{\frac{4-q}2}
\\&\quad +\frac1A \frac{M(b-1)}{b|b|^{q-m-2}}\biggl[m -2\biggl(\frac 1N+\frac 12\biggr)\frac EA
\biggr]v^{(b-1)(m-q+1)-1}z^{\frac{m+4-q}2}
\\&\quad +\frac1A\frac {M}{b|b|^{q-m-2}}\biggl[\frac m2-\biggl(\frac1N+\frac 12\biggr)\frac D{A}\biggr]v^{(b-1)(m-q+1) }z^{\frac{m-q}2}
\langle\nabla z,\nabla v\rangle
\\&\quad + \bigl(C_2+C_\varepsilon\bigr)\frac{|\nabla z|^2}{z}\quad \mbox{on}\quad \{z>0\},
\end{aligned}$$
where $A_1$ is given in \eqref{A_1}.
Estimating, similarly as in \eqref{nabla/v} we have
$$\begin{aligned}\biggl|\biggl(&\frac1N+\frac 12\biggr)\frac D{A}\frac{v^{\tau}z^{-\frac q2}}{b|b|^{q-2}}\frac{\langle\nabla z,\nabla v\rangle}A\biggr|
\\&\le \biggl(\frac1N+\frac 12\biggr)|p-2|\frac{v^{\tau}z^{1-\frac q2}}{|b|^{q-1}}\frac 1A\frac{|\nabla z||\nabla v|}z
\le \frac{\varepsilon}{A^2} v^{2\tau}z^{2-q}
+C_\varepsilon \frac{|\nabla z|^2}{z}\end{aligned}$$
and 
$$\begin{aligned}
\frac1A\biggl|\frac {M}{b|b|^{q-m-2}}\biggl[\frac m2&-\biggl(\frac1N+\frac 12\biggr)\frac D{A}\biggr]v^{(b-1)(m-q+1) }z^{\frac{m-q}2}
\langle\nabla z,\nabla v\rangle
\biggr|\\&\le
\frac1A\frac {M}{|b|^{q-m-1}}\biggl|\frac m2+\biggl(\frac1N+\frac 12\biggr)(p-2)\biggr|v^{(b-1)(m-q+1) }z^{\frac{m-q}2+1}
\frac{|\nabla z||\nabla v|}z\\&\le \frac{\varepsilon}{A^2} v^{2(b-1)(m-q+1)}z^{m+2-q}+ C_\varepsilon \frac{|\nabla z|^2}{z}
\end{aligned}$$
yielding
\begin{equation}\label{zeta_1_2}\begin{aligned} \frac12\mathscr{A}_v(z) &+ (A_1-\varepsilon)\frac{z^2}{v^2}+\frac1{A^2}\biggl(\frac1{Nb^{2(q-1)}}-\varepsilon\biggr)v^{2\tau}z^{2-q}\\&-\frac 1{A|b|^{q-2}}\biggl\{
s-2\frac{b-1}{b}\biggl(\frac 1N+\frac 12\biggr)\frac EA\biggr\}v^{\tau-1}z^{\frac{4-q}2}
\\&+\frac1{A^2}\biggl(\frac {M^2}N|b|^{2(m-q+1)}-\varepsilon\biggr)v^{2(b-1)(m-q+1)}z^{m+2-q} 
\\&
 -\frac1A \frac{M(b-1)}{b|b|^{q-m-2}}\biggl[ m-2\biggl(\frac 1N+\frac 12\biggr)\frac EA
\biggr]v^{(b-1)(m-q+1)-1}z^{\frac{m-q}2+2}
\\&+2\frac MN|b|^{m-2(q-1)}\frac1{A^2}v^{\tau +(b-1)(m-q+1)}z^{\frac{m-2q}2+2}
\le C\frac{|\nabla z|^2}{z}\quad \mbox{on}\quad \{z>0\}.
\end{aligned}\end{equation}
Put
$$\zeta_1=v^{\tau+1}z^{-\frac q2},\qquad \zeta_2=v^{(b-1)(m-q+1)+1}z^{\frac{m-q}2},$$
define
$$\mathcal H_1=\biggl(\varepsilon\frac1{A^2}\frac1{b^{2(q-1)}}\zeta_1^2+\mathcal T_\varepsilon(\zeta_1)\biggr)\frac{z^2}{v^2}, \qquad\quad \mathcal H_2=\mathcal S^M_\varepsilon(\zeta_2)\frac{z^2}{v^2}$$
with
$$\mathcal T_\varepsilon(\zeta_1)=\frac1{A^2}\biggl(\frac1{Nb^{2(q-1)}}-2\varepsilon\biggr)\zeta^2_1-\frac 1{A|b|^{q-2}}\biggl[s-2
\frac{b-1}{b}\biggl(\frac 1N+\frac 12\biggr)\frac EA\biggr]\zeta_1+{A_1}-\varepsilon$$
and
$$\mathcal S^M_\varepsilon(\zeta_2)=\frac1{A^2}\biggl(\frac {M^2}N|b|^{2(m-q+1)}-\varepsilon\biggr)\zeta_2^2+\frac1A \frac{M}{|b|^{q-m-2}}\frac{b-1}b\biggl[ 2\biggl(\frac 1N+\frac 12\biggr)\frac EA 
-m\biggr]\zeta_2,$$
so that inequality \eqref{zeta_1_2} implies
\begin{equation}\label{fin_H1_H2}\frac12\mathscr{A}_v(z)
+\mathcal H_1+\mathcal H_2\le C\frac{|\nabla z|^2}{z}\quad \mbox{on}\quad \{z>0\}.\end{equation}
By continuity we can consider  $\varepsilon=0$. We immediately note that $\mathcal H_2(t)\ge0$, since $\frac1{A^2}\frac {M^2}N|b|^{2(m-q+1)}>0$ and 
$$ \frac1A \frac{M}{|b|^{q-m-2}}\frac{b-1}b\biggl[ 2\biggl(\frac 1N+\frac 12\biggr)\frac EA 
-m\biggr]\ge0$$
by \eqref{m_bound_new}, provided that $b\ge1$.

To manage $\mathcal H_1$, we argue  as in the proof of Theorem \ref{th_sm} in the case $b\ge1$, by 
proving that 
the discriminant of $\mathcal T_\varepsilon$ is negative, namely
$$
\tilde{\mathcal D}:=\frac 1{A^2|b|^{2(q-2)}}\biggl[s-2\frac{b-1}{b}\biggl(\frac 1N+\frac 12\biggr)\frac EA\biggr]^2-\frac 4{A^2}\biggl(\frac1{Nb^{2(q-1)}}-2\varepsilon\biggr)(A_1-\varepsilon)<0.
$$
Actually we need to prove that it is possible to choose $b\ge1$ such that there exist $\varepsilon$ and $\kappa>0$
so that we have
\begin{equation}\label{discri+}
\tilde{\mathcal D}\le -\kappa\frac1{A^2}
\frac1{b^{2(q-1)}}.
\end{equation}
This produces an analogous estimate  of the form \eqref{Sep11-2}.  Consequently, choosing $\theta=\frac{2(b-1)(p-q)+2}{\tau+1}$, in particular, $\theta\in(0,2)$, and observing that, by $s>p-1$, we have
$$2\tau=2[b(s-q+1)+q-1]>2b(p-q)+2(q-p+p-1)=
2(b-1)(p-q)+2(p-1)>2(b-1)(p-q),$$
the argument used to reach \eqref{Sep7-2} can be applied with
$$\beta_2=\beta_2(b)=1-q\frac{(b-1)(p-q)+1}{b(s-q+1)+q}-p+q$$
so that
$$\lim_{b\to\infty}\beta_2=1-\frac{(p-q)(1+s)}{s-q+1}>0$$
by \eqref{beta_2_pos_sum}, thus it holds
$$\tilde{\mathcal H}_1\ge \kappa_bz^{1+\max\{1,\beta_2\}}\quad\text{for } z\ge1 \quad\text{and } \kappa_b>0.$$
We have so obtained from \eqref{fin_H1_H2}, thanks also to $\tilde{\mathcal H}_2\ge0$,
$$\frac12\mathscr{A}_v(z)
+\kappa_b z^{1+\gamma}\le C\frac{|\nabla z|^2}{z}\quad \mbox{on}\quad \{z>1\},$$
with $\gamma=\max\{1,\beta_2\}>0$.  Hence by Lemma~\ref{lem1}  and employing an argument similar to Theorem~\ref{th_HJ}, from \eqref{ineq_HJ2}, we obtain
\eqref{z-1}.

\medskip

It remains to prove \eqref{discri+}, or equivalently $A^2b^{2(q-1)}\tilde{\mathcal D}\le -\kappa$. By continuity we consider the case $\varepsilon=0$, so that we have to prove
$$
b^2\biggl[s-2\frac{b-1}{b}\biggl(\frac 1N+\frac 12\biggr)\frac EA\biggr]^2-\frac 4 NA_1\le-\kappa.
$$
From \eqref{tau}, we reach
$$b=\frac{\tau-q+1}{s-q+1}, \qquad \frac{b-1}b=\frac{\tau-s}{\tau-q+1},$$
where in particular $b\ge 1$ holds if and only if $\tau\ge s$ being $s>q-1$ by \eqref{s_bound_new}.
By replacing this expression of $b$ and using \eqref{A_1},  the above condition on the discriminant reads as follows
$$\mathscr T(\tau):=(\tau-q+1)^2\biggl[s-\frac{\tau-s}{\tau-q+1}\frac {N+2}N\frac EA\biggr]^2-\frac4N(\tau-s)(s-q+1)\biggl[\frac{\tau-s}{s-q+1}\Xi+\frac EA\biggr]<-\kappa,$$
where we recall 
$$\Xi=\frac 1N\frac{E^2}{A^2}-\frac 1A(p-q)^2\mathfrak A.$$
Now, note that $\mathscr T(\tau)$ is a trinomial of the form
$$\mathscr T(\tau)=\biggl\{\frac{[sNA-E(N+2)]^2}{N^2A^2}-\frac 4N\Xi\biggr\}\tau^2+\text{l.o.t.}$$
We claim that it results 
\begin{equation}
    \label{coeff_scrT}
    \frac{[sNA-E(N+2)]^2}{N^2A^2}-\frac 4N\Xi<0.\end{equation}
    To this aim, note that
    \eqref{coeff_scrT} is equivalent to
    $$s^2-2s\frac{N+2}N\frac EA+\frac{N+4}{N}\frac{E^2}{A^2}+\frac 4N\frac 1A(p-q)^2\mathfrak A<0$$
which is in force, by  
\eqref{estim_DA_EA}, if it holds
\begin{equation}
    \label{coeff_T_pq}s^2-2s\frac{N+2}N(q-1)+\frac{N+4}{N}(p-1)^2+\frac 4N(p-q)^2<0.\end{equation}
The above inequality is indeed valid by \eqref{s_bound_new}.
Consequently, $\lim_{\tau\to\infty}\mathcal T(\tau)=-\infty$, so it is enough to take $\tau$ sufficiently large to have $\mathscr T(\tau)\le-\kappa$, consequently $b$ will be sufficiently large yielding $b>1$.
The final Liouville property follows reasoning as in the proof of (ii) in Theorem \ref{th_sm}.
\hfill$\square$

\bigskip

{\bf Funding:} This research of M.~Bhakta
is partially supported by a DST Swarnajaynti fellowship (SB/SJF/2021-22/09) and INdAM-ICTP joint research in pairs program  for 2025. A.~Biswas is partially supported by a DST Swarnajaynti fellowship (SB/SJF/2020-21/03).
R.~Filippucci is a member of the {\em Gruppo Nazionale per l'Analisi Ma\-te\-ma\-ti\-ca, la Probabilit\`a e le loro Applicazioni} (GNAMPA) of the {\em Istituto Nazionale di Alta Matematica} (INdAM) and was  partly supported  by
 INdAM-ICTP joint research in pairs program for 2025.
M.~Bhakta and R. Filippucci would like to thank the warm hospitality of ICTP,  the travel support and daily allowances provided by INdAM-ICTP.  

\medskip

{\bf Data availability:} Data sharing not applicable to this article as no datasets were generated or analyzed during the current study.

\medskip

{\bf Conflict of interest} The authors have no conflict of interest to declare that are relevant to the content of this article.

\end{document}